\documentclass[11pt]{amsart}

\newcommand{\private}[1]{}
\usepackage{amssymb, amsmath, amscd, amsthm, color, epsfig,url}

\usepackage[all]{xy}          
\xyoption{dvips}              

\makeatletter
\renewcommand\l@subsection{\@tocline{2}{0pt}{2pc}{5pc}{}}
\makeatother

\newcommand{\R}{{\mathbb R}}

\newcommand{\abs}[1]{{\left\vert #1 \right\vert}}

\newcommand{\hofiber}{\operatorname{hofiber}}

\newcommand{\holim}{\operatorname{holim}}

\newcommand{\hocolim}{\operatorname{hocolim}}
\newcommand{\tfiber}{\operatorname{tfiber}}

\newcommand{\Map}{\operatorname{Map}}

\newcommand{\Emb}{\operatorname{Emb}}

\newcommand{\Imm}{\operatorname{Imm}}
\newcommand{\Link}{\operatorname{Link}}

\newcommand{\del}{{\partial}}

\newcommand{\Spaces}{\operatorname{Spaces}}

\theoremstyle{plain}
\newtheorem{thm}{Theorem}[section]
\newtheorem{prop}[thm]{Proposition}
\newtheorem{lemma}[thm]{Lemma}
\newtheorem{cor}[thm]{Corollary}

\theoremstyle{definition}
\newtheorem{defin}[thm]{Definition}
\newtheorem{example}[thm]{Example}
\newtheorem{def/ex}[thm]{Definition/Example}

\theoremstyle{remark}
\newtheorem{rem}[thm]{Remark}

\newcommand{\refS}[1]{Section~\ref{S:#1}}
\newcommand{\refT}[1]{Theorem~\ref{T:#1}}
\newcommand{\refC}[1]{Corollary~\ref{C:#1}}
\newcommand{\refP}[1]{Proposition~\ref{P:#1}}
\newcommand{\refD}[1]{Definition~\ref{D:#1}}
\newcommand{\refL}[1]{Lemma~\ref{L:#1}}
\newcommand{\refE}[1]{equation~$(\ref{E:#1})$}

\begin{document}


\title[Multivariable manifold calculus of functors]{Multivariable manifold calculus of functors}


\author{Brian A. Munson}
\address{Department of Mathematics, Wellesley College, Wellesley, MA}
\email{bmunson@wellesley.edu}
\urladdr{http://palmer.wellesley.edu/\~{}munson}

\author{Ismar Voli\'c}
\address{Department of Mathematics, Wellesley College, Wellesley, MA}
\email{ivolic@wellesley.edu}
\urladdr{http://palmer.wellesley.edu/\~{}ivolic}


\thanks{The second author was supported in part by the National Science Foundation grant DMS 0805406.}


\begin{abstract}
Manifold calculus of functors, due to M. Weiss, studies contravariant functors from the poset of open subsets of a smooth manifold to topological spaces. We introduce ``multivariable" manifold calculus of functors which is a generalization of this theory to functors whose domain is a product of categories of open sets. We construct multivariable Taylor approximations to such functors, classify multivariable homogeneous functors, apply this classification to compute the derivatives of a functor, and show what this gives for the space of link maps. We also relate Taylor approximations in single variable calculus to our multivariable ones.
\end{abstract}

\maketitle

\tableofcontents

\parskip=4pt
\parindent=0cm


\section{Introduction}\label{S:Intro}


The main purpose of this paper is to generalize the theory of \emph{manifold calculus of functors} developed by Weiss \cite{W:EI1} and Goodwillie-Weiss \cite{GW:EI2} (see also \cite{W:EmbCalc, GKW}) which seeks to approximate, in a suitable sense, a contravariant functor $F\colon \mathcal{O}(M)\rightarrow\Spaces$ where $M$ is a smooth compact manifold and $\mathcal{O}(M)$ the poset of open subsets of $M$. The main feature of the theory is that one can associate to $F$ another functor, called the \emph{$k^{th}$ Taylor approximation of $F$} or the \emph{$k^{th}$ stage of the Taylor tower of $F$}, given by
$$
T_kF(U)=\underset{V\in \mathcal{O}_k(U)}{\holim}\, F(V).
$$
Here $\mathcal{O}_k(U)$ is the subcategory of $\mathcal{O}(U)$ consisting of open sets diffeomorphic to at most $k$ disjoint open balls of $U$.  One then has natural transformations $F\to T_kF$ and $T_kF\to T_{k-1}F$, $k\geq0$, which combine into a \emph{Taylor tower of $F$}.  The homotopy fiber of the map $T_kF\to T_{k-1}F$, denoted by $L_kF$, is called the \emph{$k^{th}$ homogeneous layer of $F$} and is of special importance since such functors admit a classification. 

The work in this paper owes an enormous debt to Weiss' original work \cite{W:EI1} where he develops what we will in this paper call single variable manifold calculus. We will often refer the reader to that paper for details or even entire results. Although statements and proofs here are usually combinatorially more complex, many definitions and techniques used in \cite{W:EI1} carry over nicely to our multivariable setting. Manifold calculus has had many applications in the past decade \cite{ALV, LTV:Vass, M:LinkNumber, M:Emb, S:TSK, S:OKS, V:FTK}. With an eye toward extending some of them, we wish to generalize this theory to the setting where $M$ breaks up as a disjoint union of manifolds, say $M=\coprod_{i=1}^mP_i$.  The first observation is that there is an equivalence of categories 
\begin{equation}\label{E:EquivCateg}
\mathcal{O}\left(\coprod_{i=1}^mP_i\right)\cong\prod_{i=1}^m\mathcal{O}(P_i)
\end{equation}
and we may thus view an open set $U$ in $M$ as both a disjoint union $U_1\coprod \cdots\coprod U_m$ and an $m$-tuple $(U_1,\ldots, U_m)$.  Single variable manifold calculus  is already good enough to study functors $F:\mathcal{O}(\coprod_{i=1}^mP_i)\rightarrow\Spaces$, but it is useful to think of $F:\prod_{i=1}^m\mathcal{O}(P_i)\rightarrow\Spaces$ as a functor of several variables as well, and try to do calculus one variable at a time.  

The stages of the Taylor tower $T_kF$ mimic $k^{th}$ degree Taylor polynomials of an ordinary smooth function $f\colon \R\to\R$ and $L_kF$ corresponds to the homogeneous degree $k$ part of its Taylor series.  Further, $L_kF$ contains information about the analog of the $k^{th}$ derivative of $f$.  A natural place to begin our generalization of manifold calculus  to more than one variable might then be to look at the generalization of the calculus of smooth functions $f:\R\rightarrow\R$ to smooth functions $f:\R^m\rightarrow\R$.  A function is differentiable at $\vec{a}\in\R^m$ if there is a linear transformation $L:\R^m\rightarrow\R$ such that
$$
\underset{\vec{h}\rightarrow\vec{0}}{\lim}\frac{f(\vec{a}+\vec{h})-f(\vec{a})-L(\vec{h})}{|\vec{h}|}=0.
$$

One immediately is led to wondering how to find such a linear transformation $L$. It would be nice, for example, to describe $L$ as a $1\times m$ matrix. This leads to a desire to use coordinates on $\R^m$ itself, and the discovery of partial derivatives. Indeed, using the usual basis $\{e_i\}$ for $\R^m$, we can write $\vec{x}=x_1e_1+\cdots+ x_me_m$, and it is also useful to write this as a tuple $\vec{x}=(x_1,\ldots, x_m)$. One advantage of partial derivatives is that they are computed by fixing all but one of the variables: 
$$
\frac{\del f}{\del x_i}(a_1,\ldots, a_m)=\underset{h\rightarrow 0}{\lim}\frac{f(a_1,\ldots, a_i+h,\ldots, a_m)-f(a_1,\ldots, a_m)}{h}.
$$
Another nice thing about partial derivatives is that they represent the linear transformation $L$ in the form of the desired matrix. The notions of the derivative as a linear transformation and the derivative as a matrix both have their uses.

Thus  one way to think about this paper is that it introduces coordinates to the study of contravariant functors $F\colon M\rightarrow\Spaces$ where $M=\coprod_{i=1}^mP_i$. That is, view $\mathcal{O}(\coprod_{i=1}^mP_i)$ as the analog of $\R^m$, and view the equivalence of categories from \eqref{E:EquivCateg} as the analog of writing $\R^m=\R\times\cdots\times\R$. We will analogously set up a theory of calculus which allows us to treat each of the variable inputs $U_i\in\mathcal{O}(P_i)$ separately, and eventually obtain a good notion of mixed partial derivatives.

Although the importance of derivatives cannot be overstated, the philosophy of calculus of functors is centered around finding polynomial approximations and Taylor series for a given functor. It is from a good definition of polynomial that we obtain an object which deserves the name ``derivative''. This is where we differ from ordinary calculus, where one can motivate the idea of Taylor polynomials of a function $f$ by the obvious generalization of linearization.  In other words, one seeks a polynomial of a certain degree whose values and the values of whose derivatives up to a certain degree agree with those of $f$ at some point. Of course, the derivatives of $f$ determine the coefficients of the polynomial.

We will therefore begin by building polynomial approximations and obtain from them the notion of derivatives. Just as one can read off the derivatives of a function at a point by looking at the coefficients of the Taylor series, we will use the ``coefficients'' of our Taylor series to define derivatives. This having been said, it is nevertheless fairly easy to immediately give an analog of the derivative of a functor in our setting which is at least plausible. Let us consider the first and second derivatives for concreteness. The analog of difference for us is homotopy fiber, and so the analog of of $f(x+h)-f(x)$ is the following: If $U$ and $V$ are disjoint open balls, then $\hofiber(F(U\cup V)\rightarrow F(U))$ is the first derivative of $F$ at $U$. Now consider the following unorthodox formula for the second derivative of a function $f:\R\rightarrow\R$:

$$f''(x)=\underset{h_1\rightarrow 0}{\lim}\;\underset{h_2\rightarrow 0}{\lim}\frac{f(x+h_1+h_2)-f(x+h_1)-f(x+h_2)+f(x)}{h_1h_2}.$$

We draw the reader's attention to the numerator of the above expression when considering the next formula. Let $U,V_1,V_2$ be disjoint open balls. Then $$\hofiber(\hofiber(F(U\cup V_1\cup V_2)\rightarrow F(U\cup V_1))\rightarrow\hofiber(F(U\cup V_2)\rightarrow F(U)))$$

is the second derivative of $F$. The last expression can be rewritten in a way more amenable to generalization as the so-called total homotopy fiber of the commutative square (or a 2-cubical diagram)
$$
\xymatrix{
F(U\cup V_1\cup V_2) \ar[r]\ar[d] & F(U\cup V_1)\ar[d]\\
F(U\cup V_2)  \ar[r] & F(U)\\
}
$$
As a direct generalization of this, there is an analogous formula for the $n^{th}$ derivative given by the total homotopy fiber of a certain $n$-cubical diagram of spaces.


There are two main motivations and uses for the work developed in this paper.  One is to better understand the space of \emph{link maps} $\Link(P_1,\ldots, P_m;N)$, which is the space of smooth maps $f=\coprod_i f_i\colon\coprod_iP_i\rightarrow N$ such that $f_i(P_i)\cap f_j(P_j)=\emptyset$ for all $i\neq j$. We can think of this as a functor of $\mathcal{O}(\coprod_iP_i)=\prod_i\mathcal{O}(P_i)\rightarrow\Spaces$, and so it is a functor of several variables. This space has been studied by many \cite{CR:LinkingInv, HabLin-Classif, HQ:Bordism, Kosch-Milnor, Milnor-Mu, ST:HOIN, Scott:HoLinks, Skop:Massey-Rolfsen} and the first author has in fact already applied Weiss' manifold calculus to it in \cite{M:LinkNumber} (see also \cite{GM:LinksEstimates}).  Exploring the connection further may in particular lead to a new (and more conceptual) proof of the Habegger-Lin classification of homotopy string links \cite{HabLin-Classif} and provide a new framework for Koschorke's generalizations of Milnor invariants \cite{Kosch-Milnor}.

The other motivation is the study of embeddings and link maps of $\coprod_i\R$ in $\R^n$, $n\geq 3$, i.e.~the study of (long) links and homotopy links.  Manifold calculus was used very effectively in the study of embeddings of $\R$ in $\R^n$ and the idea is to generalize many of the results obtained in that case using multivariable calculus.  For example, it was shown in \cite{V:FTK} that the single variable Taylor tower for long knots in $\R^3$ classifies finite type knot invariants.  This relied on a construction of a cosimplicial model arising from the Taylor tower \cite{S:TSK} and the associated spectral sequences.   In \cite{MV:Links}, we give multi-cosimplicial analogs for links and link maps and deduce an analogous results, namely that the multivariable Taylor tower contains all finite type invariants of links and homotopy links.   This is in turn expected to lead to a way of recognizing  classical Milnor invariants in the multivariable Taylor tower.  The crucial ingredient in \cite{MV:Links} is the finite model for the multivariable Taylor tower from \refS{FiniteModels}.  



\subsection{Organization of the paper}\label{S:Organization}


This paper is organized as follows:

\textbullet\ \ In \refS{Conventions}, we set some notational conventions and state the definitions used throughout the paper.

\textbullet\ \ In \refS{CalculusReview}, we survey some of the main results of \cite{W:EI1}, with an emphasis on the results we desire to generalize to the multivariable setting. We include a few examples in \refS{SetupExamples} and introduce the language of cubical diagrams in \refS{Cubes}. We review the main definitions and results about polynomial functors and the stages $T_kF$ of the Taylor tower in \refS{Polynomials}.    The convergence of the Taylor tower for the embedding functor (\refT{gwemb}), which is the main example in the theory, is recalled in \refS{Convergence}.  Homogeneous functors and their classification (\refT{homogclass}) are reviewed in \refS{Homogeneous}.  Finally, we discuss some technicalities regarding the passage to manifolds with boundary in \refS{ManifoldsBoundary} in the cases we ultimately care most about, namely embeddings and link maps.

\textbullet\ \ In \refS{MultivarCalculus}, we develop the analogs of the results discussed in \refS{Polynomials} and \refT{gwemb}.  We begin by discussing the equivalence of categories from equation \eqref{E:EquivCateg} in \refS{OpenSets}.  In \refS{MultiPolynomials}, we define the notion of a polynomial functor of  multidegree $\vec{\jmath}=(j_1, ..., j_m)$ (\refD{Multipoly}), give some examples, define the multivariable Taylor approximations to a functor (\refD{jMultiStage}) and its Taylor multi-tower, and show the approximations satisfy certain properties.  \refT{Multipolyclass} gives criteria for checking when two polynomials of the same multidegree are equivalent functors; this has applications in \refS{multihomog}. We then give a multivariable analog of \refT{gwemb} in \refS{MultiConvergence}.

\textbullet\ \ In \refS{multihomog}, which is the analog of \refS{Homogeneous}, we classify multivariable homogeneous functors.  In \refS{multihomogdef}, we discuss what one can deduce  from \refT{homogclass} which classifies homogeneous functors in single variable calculus.  This is the content of \refP{homogdisjoint}.  We then use this to motivate our definition of homogeneous multivariable polynomials (\refD{MultiHomogeneous}).  The goal then is to prove a classification theorem for such polynomials in \refS{Classification}  (\refT{multihomogclass}), but to do so we need to complete some preliminary work on homogeneous functors that look like spaces of sections of some fibration in \refS{Sections}.  It is precisely the fibers of this fibration which deserve to be called the mixed partial derivatives of a functor.  Finally, in \refS{fibersclasslink} we work out the fibers of the classifying fibration from \refT{multihomogclass}.

\textbullet\ \ In \refS{TowersRelation}, we compare the single variable and multivariable Taylor towers and their stages. \refT{twotowers} essentally tells us how to put the Taylor approximations in the multivariable setting together to obtain the Taylor approximations in the single variable setting.  The latter theorem uses 
the classification of homogeneous multivariable functors (\refT{multihomogclass}).

\textbullet\ \ In \refS{FiniteModels},  we give a non-functorial finite model for the Taylor approximations in the case where $F$ is the embedding or link maps functor and $M=\coprod_i I$ is a disjoint union of intervals (so $F$ is the space of string links or homotopy string links respectively). These are precisely the functors which will be studied in greater detail in \cite{MV:Links}.


\subsection{Acknowledgements}  The authors would like to thank Tom Goodwillie and Victor Turchin for helpful conversations, and the referee for a thorough report.


\section{Conventions}\label{S:Conventions}


Throughout the paper, we will assume the reader is familiar with homotopy limits and colimits.  Following Weiss \cite{W:EI1}, our spaces will be fibrant simplicial sets. Our examples are all in one way or another related to the space of maps $\Map(X,Y)$, which has a simplicial structure as follows: a $k$-simplex is a fiber-preserving map $f_k:X\times\Delta^k\rightarrow Y\times\Delta^k$; that is, $f_k(x,s)=(y,s)$. If $X$ and $Y$ are smooth manifolds, then $f_k$ should be smooth.  Other conventions are as follows.

{\bf Sets:}

\textbullet\ \ 
For a nonnegative integer $k$, let $[k]$ denote the ordered set $\{0,1,\ldots, k\}$, and let $\underline{k}$ denote the set $\{1,2,\ldots, k\}$. 

\textbullet\ \ 
For a finite set $S$, we let $\abs{S}$ stand for its cardinality. 

\textbullet\ \ 
Let $\mathcal{P}(S)$ stand for the poset of all subsets of $S$, and $\mathcal{P}_0(S)$ the subposet of all non-empty subsets of $S$. For a tuple $\vec{S}=(S_1,\ldots, S_m)$ of finite sets, let $\mathcal{P}(\vec{S})=\prod_{i=1}^m\mathcal{P}(S_i)$. Let $\mathcal{P}_0([\vec{\jmath}])=\mathcal{P}_0([j_1])\times\cdots\times\mathcal{P}_0([j_m])$ be the full subposet of those $\vec{S}=(S_1,\ldots, S_m)$ such that $S_i\neq\emptyset$ for all $i$.

\textbullet\ \ 
For tuples $\vec{\jmath}=(j_1,j_2,\ldots, j_m)$ and $\vec{\jmath}'=(j'_1,j'_2,\ldots, j'_m)$ of integers, we say $\vec{\jmath}\leq\vec{\jmath}'$ if $j_i\leq j'_i$ for all $1\leq i\leq m$, and $\vec{\jmath}<\vec{\jmath}'$ if $\vec{\jmath}\leq\vec{\jmath}'$ and there exists $i$ such that $j_i<j_i'$. 

\textbullet\ \ 
Let $|\vec{\jmath}|=\sum_i\abs{j_i}$, and $|\vec{\jmath}-\vec{k}|=\sum_i\abs{j_i-k_i}$. 

\textbullet\ \ 
Let $\mathcal{Z}^m$ denote the poset whose objects are $m$-tuples $\vec{\jmath}=(j_1,\ldots, j_m)$ of non-negative integers, and with $\vec{k}\leq\vec{\jmath}$ if $k_i\leq j_i$ for all $i$. Let 
$\mathcal{Z}^m_{\leq\vec{\jmath}}$ (respectively $\mathcal{Z}^m_{<\vec{\jmath}}$) denote the subcategory of all $\vec{k}$ which satisfy $\vec{k}\leq\vec{\jmath}$ (respectively $\vec{k}<\vec{\jmath}$). Let 
$\mathcal{Z}^m_{\leq k}$ (respectively $\mathcal{Z}^m_{<k}$) denote the subcategory of all $\vec{\jmath}$ which satisfy $|\vec{\jmath}|\leq k$ (respectively $|\vec{\jmath}|< k$).

{\bf Spaces:}

\textbullet\ \ 
For a space $X$ and a nonnegative integer $k$, let $C(k,X)$ be the configuration space of $k$ points in $X$, in other words
$$C(k,X)=X^k-\Delta_{fat}(X^k),$$ where 
$$\Delta_{fat}(X^k)=\{(x_1,x_2,\ldots, x_k)|x_i=x_j\mbox{ for some }i\neq j\}$$ is the \emph{fat diagonal}. When $X$ and $k$ are understood, the fat diagonal will be denoted simply by $\Delta_{fat}$. 

\textbullet\ \ 
Let $sp_kX=X^k/\Sigma_k$ denote the $k$th symmetric product of  $X$.

\textbullet\ \ 
Let ${X\choose k}=C(k,X)/\Sigma_k$ be the quotient of $C(k,X)$ by the free action of the symmetric group $\Sigma_k$ which permutes the coordinates.  This is the space of unordered configurations of $k$ points in $X$. It is the complement of the image of the fat diagonal in $sp_kX$. We will also denote the image of $\Delta_{fat}$ in $sp_kX$ by $\Delta_{fat}$, which should cause no confusion since it will always be clear from context what we mean.

\textbullet\ \
For $\vec{X}=(X_1,\ldots, X_m)$ and a tuple of nonnegative integers $\vec{\jmath}=(j_1\ldots, j_m)$, let $sp_{\vec{\jmath}}\vec{X}=\prod_isp_{j_i}X_i$

\textbullet\ \ 
For an $m$-tuple of spaces $\vec{X}=(X_1,\ldots, X_m)$ and an $m$-tuple of nonnegative integers $\vec{\jmath}=(j_1,\ldots, j_m)$, let 
$$
\Delta_{\vec{fat}}(\vec{X}^{\vec{\jmath}})=\{\vec{x}=(x_1,\ldots, x_m)\, |\, x_i\in X_i^{j_i} \text{ and there exists $i$ with $x_i\in\Delta_{fat}(X^{j_i}_i)$}\}.
$$ 
Again, when $\vec{X}$ and $\vec{\jmath}$ are understood, we will denote this simply by $\Delta_{\vec{fat}}$. Let 
$$
\Delta_{fat_i}(\vec{X}^{\vec{\jmath}})=X_1^{j_1}\times\cdots\times X_{i-1}^{j_{i-1}}\times\Delta_{fat}(X_i^{j_i})\times X_{i+1}^{j_{i+1}}\times\cdots\times X_{m}^{j_{m}}.
$$ 
Thus we have 
$$
\Delta_{\vec{fat}}(\vec{X}^{\vec{\jmath}})=\cup_{i=1}^m\Delta_{fat_i}(\vec{X}^{\vec{\jmath}}).
$$

\textbullet\ \ 
Let $\vec{X}\choose\vec{\jmath}$ denote the product $\prod_i{X_i\choose j_i}$, and let $sp_{\vec{\jmath}}\vec{X}=\prod_{i}sp_{j_i}X_i$.

{\bf Categories and functors:}

\textbullet\ \ We let $\Spaces$ denote the category of fibrant simplicial sets, and $\Spaces_\ast$ the category of based fibrant simplicial sets.

\textbullet\ \ 
When we say ``functor", we will mean either a covariant or a contravariant one. When it is not clear from the context or useful to point out, we will specify the variance.


\section{Review of the single variable manifold calculus of functors}\label{S:CalculusReview}


We begin with a survey of the main results of manifold calculus. Details can be found in \cite{W:EI1, GW:EI2}. We will generalize many of these results to the multivariable setting in Section \ref{S:MultivarCalculus}. 


\subsection{Setup and examples}\label{S:SetupExamples}


Let $M$ be a smooth compact manifold, and $\mathcal{O}(M)$ the poset of open subsets of $M$ with inclusions as morphisms. Manifold calculus studies contravariant functors
$$F:\mathcal{O}(M)\longrightarrow\mathcal{C},
$$
where $\mathcal{C}$ is usually the category of spaces or spectra.  For us, $\mathcal{C}$ will always be $\Spaces$ or $\Spaces_\ast$.

\begin{defin}
Let $M$ be as above, and let $U,V\in\mathcal{O}(M)$ with $U\subset V$. The inclusion $i:U\hookrightarrow V$ is an \emph{isotopy equivalence} if there exists an embedding $e:V\rightarrow U$ with the property that $i\circ e$ and $e\circ i$ are isotopic to $1_V$ and $1_U$ respectively.
\end{defin}

One may think of an isotopy equivalence $U\hookrightarrow V$ as a ``thickening'' of $U$. The functors we study are required to satisfy the following two axioms.

\begin{defin}\label{D:goodfunctor}
A contravariant functor $F:\mathcal{O}(M)\longrightarrow\Spaces$ is \emph{good} if 
\begin{enumerate}
\item It takes isotopy equivalences to homotopy equivalences, and 
\item For any nested sequence $\{U_i\}$ of open subsets of $M$, the map 
$$ 
F (\cup U_i) \longrightarrow \underset{i}{\holim}\, \, F(U_i)
$$
is a weak equivalence.
\end{enumerate}
\end{defin}

The second part of this definition says that the values of a good functor are  determined by their values on the interiors of compact codimension zero submanifolds of $M$. Since one typically is only interested in the values on such open sets, one could start with functors defined on the category of open sets in $M$ which are the interiors of compact codimension zero submanifolds, with inclusion maps as the morphisms, and extend along the inclusion of this category into $\mathcal{O}(M)$. In this sense, the second axiom is optional. The reader may freely choose either according to her tastes. Here are some examples of good functors, where the open set $U\subset M$ is the variable.


\begin{def/ex}\label{DE:MainExamples}\ \ 
Let $X$ be any space, $M$ and $N$ be smooth manifolds with $M$ compact, and let $U\in\mathcal{O}(M)$.
\begin{itemize}
\item  $\Map(U, X)$, the space of maps from $U$ to $X$;
\item  $\Map\left({U\choose k}, X\right)$;
\item  $\Imm(U,N)$, the space of immersions of $U$ into $N$;
\item  $\Emb(U,N)$, the space of embeddings of $U$ into $N$;
\item  $\overline{\Emb}(U,N)=\hofiber(\Emb(U,N)\hookrightarrow \Imm(U,N))$;
\item $\Link(U_1, U_2, ..., U_m; N)$, the space of link maps from $\coprod_{i=1}^n U_i$ to $N$ (i.e.~smooth maps such that the images of the $U_i$ are disjoint). Here $M=\coprod_{i=1}^mP_i$, the $P_i$ are of dimension $p_i$, and $U_i\in\mathcal{O}(P_i)$.
\end{itemize}
\end{def/ex}

All of these are contravariant functors since an inclusion of open subsets of $M$ gives rise to a restriction map. The last two functors are of most concern to us, specifically in the case where $M$ is one-dimensional \cite{MV:Links}. Before we continue, it will be useful to survey some definitions and basic results about cubical diagrams, a useful organizational tool central to calculus of functors.

\begin{rem}
A word about basepoints is in order. We will assume all of our functors are based in the sense that, using the setup as above, $F(M)$ has a preferred basepoint, which then bases $F(U)$ for all $U$ via the map $F(M)\to F(U)$. In particular examples this may or may not be possible. For instance, one might be interested in the functor $F(U)=\Emb(U,N)$, where $U$ is an open subset of $M$, and hope to build an element of $F(M)$ from elements of $F(U)$ for various $U$. It is possible for $F(M)$ to be empty even if $F(U)$ are non-empty for many choices of $U$. For instance, consider the functor $U\mapsto\Emb(U,\R^2)$ for $U\subset S^2$. Unfortunately, the most useful machinery, namely the classification of homogeneous functors, requires a choice of basepoint in $F(M)$ to work smoothly, and while one can still achieve partial results with weaker hypotheses, it streamlines the theory a great deal to assume a basepoint exists. Similar comments apply to the multivariable functors which we consider in this paper.
\end{rem}


\subsection{Cubical diagrams and homotopy limits}\label{S:Cubes}


The purpose of this section is to remind the reader of some relevant definitions and a few useful results.  Details about cubical diagrams can be found in \cite[Section 1]{CalcII}. 

\begin{defin}
Let $T$ be a finite set. A \textsl{$\abs{T}$-cube} of spaces is a covariant functor $$\mathcal{X}\colon\mathcal{P}(T)\longrightarrow\Spaces.$$
\end{defin}

\begin{defin}
A $|T|$-cube of spaces is \emph{$k$-cartesian} if the map

$$
\mathcal{X}(\emptyset)\longrightarrow\underset{S\neq\emptyset}{\holim}\, \, \mathcal{X}(S)
$$

is $k$-connected. In case $k=\infty$, i.e.~ if the map is a weak equivalence, we say the cube $\mathcal{X}$ is \textsl{homotopy cartesian}. Dually, it is \emph{$k$-cocartesian} if the map

$$
\underset{S\neq T}{\hocolim}\, \, \mathcal{X}(S)\longrightarrow\mathcal{X}(T)
$$

is $k$-connected, and \emph{homotopy cocartesian} if $k=\infty$.
\end{defin}

\begin{rem}
When $|T|=2$, the notions of homotopy cartesian (resp.~ homotopy cocartesian) square and homotopy pullback (resp.~ homotopy pushout) square agree.
\end{rem}

\begin{defin}
The \textsl{total homotopy fiber}, or \textsl{total fiber}, of a $\abs{T}$-cube $\mathcal{X}$ of based spaces, denoted $\tfiber(S\mapsto \mathcal{X}(S))$ or $\tfiber(\mathcal{X})$, is the homotopy fiber of the map 
$$
\mathcal{X}(\emptyset)\longrightarrow\underset{S\neq\emptyset}{\holim}\, \, \mathcal{X}(S).
$$ 
\end{defin}

The total fiber can also be thought of as an iterated homotopy fiber. That is, view a $\abs{T}$-cube $\mathcal{X}$ as a map (i.e.~a natural transformation) of $(\abs{T}-1)$-cubes $\mathcal{Y}\rightarrow\mathcal{Z}$. In this case, $\tfiber(\mathcal{X})=\hofiber(\tfiber(\mathcal{Y})\rightarrow\tfiber(\mathcal{Z}))$. More precisely, we have

\begin{prop}\label{P:IteratedHofiber}
Suppose $\mathcal{X}$ is a $\abs{T}$-cube of based spaces, and let $t\in T$. Then there is an equivalence $$\tfiber(\mathcal{X})\simeq\hofiber(\tfiber(S\mapsto \mathcal{X}(S))\longrightarrow\tfiber(S\mapsto \mathcal{X}(S\cup\{t\})),$$ 
where $S$ ranges through subsets of $T-\{t\}$.
\end{prop}


\begin{prop}\label{P:juxtaposecubes}
Let $\mathcal{X}$ and $\mathcal{Y}$ be $k$-cartesian $\abs{T}$-cubes of based spaces, and suppose that, for some $t\in T$, $\mathcal{X}(S\cup\{t\})=\mathcal{Y}(S)$ for all $S\subset T-\{t\}$. Then the $\abs{T}$-cube $\mathcal{Z}$ defined by $\mathcal{Z}(S)=\mathcal{X}(S)$ and $\mathcal{Z}(S\cup\{t\})=\mathcal{Y}(S\cup\{t\}$ for $S\subset T-\{t\})$ is also $k$-cartesian.
\end{prop}

\begin{proof}
View $\mathcal{X}$ and $\mathcal{Y}$ as $1$-cubes of $(\abs{T}-1)$-cubes $(S\mapsto \mathcal{X}(S))\rightarrow (S\mapsto \mathcal{X}(S\cup\{t\}))$ and$(S\mapsto \mathcal{Y}(S))\rightarrow (S\mapsto \mathcal{Y}(S\cup\{t\}))$ respectively, where $S\subset T-\{t\}$. By hypothesis, the maps
$$
\tfiber(S\mapsto \mathcal{X}(S))\rightarrow\tfiber(S\mapsto \mathcal{X}(S\cup\{t\}))$$ and $$\tfiber(S\mapsto \mathcal{Y}(S))\rightarrow\tfiber(S\mapsto \mathcal{Y}(S\cup\{t\})
$$ 
are $k$-connected, and hence, since $(S\mapsto \mathcal{X}(S\cup\{t\}))=(S\mapsto \mathcal{Y}(S))$, the composed map 
$$\tfiber(S\mapsto \mathcal{X}(S))\rightarrow \tfiber(S\mapsto \mathcal{Y}(S\cup\{t\})
$$ 
is $k$-connected, and therefore $\mathcal{Z}$ is $k$-cartesian.
\end{proof}

We end with a result which is not about cubical diagrams, but which is useful for studying functors of more than one variable. It can be found in \cite[Ch. XI, Example 4.3]{BK}.  The statement clearly generalizes to the product of more than two categories. 

\begin{prop}\label{P:holimproduct}
If $F:\mathcal{C}_1\times\mathcal{C}_2\rightarrow\Spaces$ is a bifunctor, then there are homeomorphisms $$\underset{\mathcal{C}_2}{\holim}\, \,\underset{\mathcal{C}_1}{\holim}\, \, F\cong\underset{\mathcal{C}_1\times\mathcal{C}_2}{\holim}\, \, F\cong\underset{\mathcal{C}_1}{\holim}\, \,\underset{\mathcal{C}_2}{\holim}\, \, F.$$
\end{prop}

\subsection{Polynomial functors and Taylor tower}\label{S:Polynomials}


Manifold calculus seeks to approximate a good functor $F$ by a sequence of functors $T_kF$ which are ``polynomial of degree $\leq k$'', and are the analogs of the $k^{th}$ degree Taylor approximations to a function in ordinary calculus.

\begin{defin}\label{D:Poly}
A good functor $F:\mathcal{O}(M)\rightarrow\Spaces$ is said to be  \textsl{polynomial of degree $\leq k$} if for all $U\in\mathcal{O}(M)$ and for all pairwise disjoint non-empty closed subsets $A_0,A_1,\ldots A_k$ of $U$, the map 
$$F(U)\longrightarrow\underset{S\neq\emptyset}{\holim}\, \, F(U_S),$$ where $U_S=U-\cup_{i\in S}A_i$, is a homotopy equivalence.
\end{defin}

An ordinary polynomial of degree $k$ in a single variable is, of course, determined by its values on $k+1$ distinct points, which is what this definition is attempting to mimic. 

\begin{example}
The functor $U\mapsto\Map(U,X)$ is polynomial of degree $\leq 1$ for any space $X$, as is $U\mapsto\Imm(U,N)$. The functor $U\mapsto\Map\left({U\choose k},X\right)$ is polynomial of degree $\leq k$. See \cite[Examples 2.3 and 2.4]{W:EI1}.
\end{example}

\begin{defin}
For $k\geq0$, let $\mathcal{O}_k(M)$ denote the full subcategory of open sets diffeomorphic to at most $k$ disjoint open balls. 
\end{defin}
Now we can define the $k^{th}$ degree polynomial functors $T_kF$.

\begin{defin}\label{D:kStage}  Define the \emph{$k^{th}$ degree Taylor approximation of $F$} to be 
$$T_kF(U)=\underset{\mathcal{O}_k(U)}{\holim}\, \, F.$$
\end{defin}

 The $T_kF$ themselves are good functors. Note that $T_0F\simeq F(\emptyset)$.

Also note that if $U\in\mathcal{O}_k(M)$, then $U$ is a final object in $\mathcal{O}_k(U)$ and so (since $F$ is contravariant) the natural map $F(U)\rightarrow T_kF(U)$ is an equivalence. Hence $F$ and $T_kF$ agree when the input is at most $k$ open balls. 

It is clear that there are canonical maps $F\rightarrow T_kF$ and $T_kF\rightarrow T_{k-1}F$ for all $k\geq1$, and that they are all compatible in the sense that the obvious diagrams commute.  Thus we can make the following definition.

\begin{defin}\label{D:SingleTaylorTower}
The \emph{Taylor tower of $F$} is the sequence of functors 
$$
F\longrightarrow\left(T_0F\leftarrow T_1F\leftarrow\cdots\leftarrow T_kF\leftarrow\cdots\leftarrow T_{\infty}F    \right)
$$
(pictured with the canonical map from $F$) where
$$
T_{\infty}F = \underset{k}{\holim}\, T_kF.
$$
Polynomial approximation $T_kF$ will also be called  the \emph{$k^{th}$ stage} of the Taylor tower of $F$.
\end{defin}

In our subsequent work \cite{MV:Links}, we will work with certain ``smaller" models for stages $T_kF$ in some special cases.  These will be defined in Section \ref{S:FiniteModels}.

It is somewhat tedious to show that $T_k F$ is polynomial of degree $\leq k$ (stated below with a couple of other facts as \refT{tkpoly}). The basic idea is to modify the definition of $\mathcal{O}_k(M)$ with respect to some open cover of $M$ so that each component of an element $U\in\mathcal{O}_k(M)$ is contained in some set in the open cover. One shows that the homotopy limit in $\refD{kStage}$ is equivalent to a homotopy limit over the category of open balls whose components are contained in a set in an open cover of $M$, and when the pairwise disjoint closed sets $A_0,\ldots, A_k$ are chosen, selects an open cover such that none of the open sets in the cover meet more than one $A_i$. Then the pigeonhole principle implies that each $U$ of at most $k$ open balls must miss some $A_i$. Finally one applies the following technical lemma, which we will require in \refP{mcubehomog} and \refT{twotowers}, to finish the proof.

\begin{defin}\label{D:ideal}
An \emph{ideal} in a poset $\mathcal{Q}$ is a subset $\mathcal{R}$ of $\mathcal{Q}$ such that for all $b\in\mathcal{R}$, $a\in\mathcal{Q}$ with $a\leq b$ implies $a\in\mathcal{R}$.
\end{defin}

\begin{lemma}[\cite{W:EI1}, Lemma 4.2]\label{L:posetideals}
Suppose the poset $\mathcal{Q}$ is a union of ideals $\mathcal{Q}_i$ with $i\in T$, where $T$ is some finite set. For a finite nonempty subset $S\subset T$, let $\mathcal{Q}_S=\cap_{i\in S}\mathcal{Q}_i$. Let $E$ be a contravariant functor from $\mathcal{Q}$ to $\Spaces$. Then the canonical map

$$\underset{\mathcal{Q}}{\holim}\, E\longrightarrow\underset{S\neq\emptyset}{\holim}\, \underset{\mathcal{Q}_S}{\holim}\, E$$

is an equivalence. 
\end{lemma}

The following important statement, which characterizes polynomials, is Theorem 5.1 in \cite{W:EI1}.

\begin{thm}\label{T:polyclass}
Let $k$ be a nonnegative integer, and let $F\rightarrow G$ be a natural transformation between polynomials of degree $\leq k$. If $F(U)\rightarrow G(U)$ is an equivalence for all $U\in\mathcal{O}_k(M)$, then it is an equivalence for all $U\in\mathcal{O}(M)$.
\end{thm}

Here is an easy but useful corollary, which can be used to prove part (2) of \refT{tkpoly}.

\begin{cor}\label{C:polyiffTk}
A good functor $F$ is polynomial of degree $\leq k$ if and only if the canonical map $F(U)\to T_kF(U)$ is an equivalence for all $U\in\mathcal{O}(M)$.
\end{cor}

\begin{proof}
If $F$ is polynomial of degree $\leq k$, then by remarks following \refD{kStage}, whenever $U\in\mathcal{O}(M)$ is a disjoint union of at most $k$ open balls, the map $F\to T_kF$ is an equivalence. It follows from \refT{polyclass} that $F(V)\to T_kF(V)$ is an equivalence for all $V\in\mathcal{O}(M)$. If $F\to T_kF$ is an equivalence, then since $T_kF$ is polynomial of degree $\leq k$, so must $F$ be, since homotopy limits preserve equivalences.
\end{proof}

The first two parts of the following are Theorem 6.1 in \cite{W:EI1}.  The third follows immediately from the definition of a polynomial functor and standard facts about cubical diagrams which can be found in \cite[Section 1]{CalcII}.

\begin{thm}\label{T:tkpoly}\ \ 
Let $k$ be a nonnegative integer.
\begin{enumerate}
\item $T_kF$ is polynomial of degree $\leq k$.
\item If $F$ is polynomial of degree $\leq k$, then $F\rightarrow T_kF$ is an equivalence.
\item If $F$ is polynomial of degree $\leq k$, then it is polynomial of degree $\leq l$ for all $k\leq l$.
\end{enumerate}
\end{thm}

The following proposition establishes a universal property of the functors $T_kF$.

\begin{prop}\label{P:factorthru}
Let $F,G:\mathcal{O}(M)\rightarrow\Spaces$ be good functors. Suppose $G$ is polynomial of degree $\leq k$. Then a natural transformation $F\rightarrow G$ factors through $T_kF$.
\end{prop}

\begin{proof}
The transformation $F\rightarrow G$ induces a map $T_kF\rightarrow T_kG$, and since $G\rightarrow T_kG$ is an equivalence, we can use a homotopy inverse to this to obtain a transformation $T_kF\rightarrow G$.
\end{proof}

One other fact which is useful is the following

\begin{prop}\label{P:iteratedT}
$T_jT_kF\simeq T_{\min\{j,k\}}F\simeq T_kT_jF$
\end{prop}

\begin{proof}
Without loss of generality assume $j\leq k$.

For $U\in\mathcal{O}(M)$, we have 

$$T_jT_kF(U)=\underset{V\in\mathcal{O}_j(U)}{\holim}\, \, T_kF(V).$$

Of course, $T_kF(V)=\holim_{W\in\mathcal{O}_k(V)}F(W)$, and since $j\leq k$, then $V\in\mathcal{O}_k(V)$ is a final object, and hence $T_kF(V)\simeq F(V)$, and hence $T_jT_kF\simeq T_jF$. To show that $T_jT_k F\simeq T_kT_j F$, note that they are both polynomials of degree $\leq k$. By \refT{polyclass}, it is enough to verify they have the same values for $U\in\mathcal{O}_k$. On the one hand, since $U$ is final in $\mathcal{O}_k(U)$, we have $T_jT_kF(U)\simeq T_jF(U)$. On the other hand, $T_kT_jF(U)=T_k(T_jF)(U)\simeq T_jF(U)$ for the same reason.
\end{proof}





\subsection{Convergence of the Taylor tower for the embedding functor}\label{S:Convergence}


An important example in manifold calculus is the functor $U\mapsto\Emb(U,N)$ and its approximations $T_k\Emb(U,N)$. In fact, manifold calculus is often referred to in the literature as \emph{embedding calculus}. In this case, we have the following result of Goodwillie and Klein. 


\begin{thm}[\cite{GK}]\label{T:gwemb} Let $M$ and $N$ be smooth manifolds of dimensions $m$ and $n$ respectively.  Then the map $\Emb(M,N)\rightarrow T_k\Emb(M,N)$ is $(k(n-m-2)+1-m)$-connected.  In particular, if $n-m>2$, the map
$$
\Emb(M,N)\longrightarrow T_{\infty}\Emb(M,N)
$$
is an equivalence.
\end{thm}

We can refine this slightly: If  $M$ is replaced by the interior $U$ of a codimension $0$ submanifold $L$ which has a handlebody decomposition with handles of index at most $l$, then $\Emb(U,N)\rightarrow T_k\Emb(U,N)$ is $(k(n-l-2)+1-l)$-connected. Among other things, the proof of this theorem involves an induction on the handle index of such submanifolds.


\subsection{Homogeneous functors}\label{S:Homogeneous}


Of special interest are homogeneous functors, for they admit a classification theorem (\refT{homogclass} below). For a good functor $F$, choose a basepoint in $F(M)$, assuming one exists. This endows $F(U)$ with a basepoint for all $U\in\mathcal{O}(M)$.

\begin{defin}\label{D:Homogeneous}
A good functor $F\colon \mathcal{O}(M)\rightarrow\Spaces$ is \textsl{homogeneous of degree  $k$} if it is polynomial of degree $\leq k$ and $T_{k-1}F(U)$ is contractible for all $U\in\mathcal{O}(M)$.
\end{defin}

\begin{defin}\label{D:Layers}
Let $F$ be a good functor, and choose a basepoint in $F(M)$. Define the \emph{$k^{th}$ layer of the Taylor tower of $F$} to be 
$$L_kF=\hofiber(T_kF\longrightarrow T_{k-1}F).$$ 
\end{defin}

\begin{thm}[\cite{W:EI1}, Theorem 8.5]\label{T:homogclass}\ \ 
\begin{itemize}
\item The $k^{th}$ layer $L_kF$ is homogeneous of degree $k$.
\item Let $E$ be homogeneous of degree $k$. Then there is an equivalence, natural in $U$, 
$$E(U)\longrightarrow \Gamma^c\left({U\choose k},Z;p\right),$$ 
where $\Gamma^c$ denotes the space of sections supported away from the diagonal (i.e.~equal to a given section in some neighborhood of the fat diagonal) of a fibration $p:Z\rightarrow {U\choose k}$.
\end{itemize}
\end{thm}

The fiber over $S$ of the fibration $p$ are the total homotopy fiber of a $k$-cube of spaces made up of the values of $E$ on a tubular neighborhood of $S$. With \refS{multihomog} in mind, it is useful to be a bit more explicit about the meaning of ``compactly supported sections''. Define

$$\Gamma\left(\del{U\choose k},Z;p\right)=\underset{Q\in\mathcal{N}}{\hocolim}\;\Gamma\left({U\choose k}\cap Q,Z;p\right),$$

where $\mathcal{N}$ is the category whose objects are neighborhoods $Q$ of $\Delta_{fat}$ in $sp_kU$, and whose morphisms are the inclusion maps. Then we define

\begin{equation}\label{E:bdysections}
\Gamma^c\left({U\choose k},Z;p\right)=\hofiber\left(\Gamma\left({U\choose k},Z;p\right)\rightarrow\Gamma\left(\del{U\choose k},Z;p\right)\right).
\end{equation}


\subsection{Manifolds with boundary}\label{S:ManifoldsBoundary}


Everything said so far can be extended to the case where $M$ has boundary. Here we wish to focus on the two functors we care about most, the space of embeddings and the space of link maps, which we will return to in \refS{FiniteModels}.  Suppose $M$ is a smooth compact manifold (possibly with many components, with possibly varying dimensions), with boundary $\del M$ and let $\mathcal{O}(M)$ denote the poset of open subsets of $M$ which contain $\del M$. Further suppose that $M$ is a neat submanifold of a manifold $N$ with boundary; in other words, $M$ meets $\del N$ transversely and $\del M=M\cap \del N$ (for details, see, for example,  Definition 2.2 in \cite{Kos:DiffMflds}). 

\begin{defin}\label{D:BoundaryFunctors}\ \ 
\begin{itemize}
\item For $U$ in $\mathcal{O}(M)$, let $\Emb_{\del}(U,N)$ be the space of smooth embeddings of $U$ in $N$ which agree with the inclusion $M\hookrightarrow N$ near $\del M$. 
\item If $M=\coprod_i P_i$ is a disjoint union, for $U\in\mathcal{O}(M)$ let $U_i=U\cap P_i$. We define $\Link_\del(U_1,\ldots, U_n;N)$ to be the space of link maps of $\coprod_i U_i\rightarrow N$ which agree with the inclusion $M\hookrightarrow N$ near $\del M$.
\end{itemize}
\end{defin}

 Both of these functors are good. 
 Of most interest to us is the case where $P_i=I$ for all $i$ and $N=\R^{n-1}\times I$, which will be studied carefully in \cite{MV:Links}.

\section{Multivariable manifold calculus of functors}\label{S:MultivarCalculus}


We now generalize the definition of the Taylor tower and several results from the previous section to the setting where $M$ is a disjoint union of manifolds by developing a calculus which allows us to treat each of the variables separately.


\subsection{Categories of open sets}\label{S:OpenSets}


Let $P_1,\ldots, P_m$ be smooth compact manifolds, and consider the poset $\mathcal{O}(\coprod_{i}P_i)$ of open subsets of $\coprod_{i}P_i$. There is an equivalence of categories 
\begin{align*}
\mathcal{O}\left(\coprod_{i}P_i\right) & \longrightarrow \prod_i\mathcal{O}(P_i) \\ 
    U & \longmapsto (U\cap P_1, \ldots, U\cap P_m).
\end{align*} 
 We will often denote elements $(U_1,\ldots, U_m)\in\prod_i\mathcal{O}(P_i)$ by $\vec{U}$. As before, $\mathcal{O}_k(\coprod_{i}P_i)$ is the full subcategory of $\mathcal{O}(\coprod_{i}P_i)$ whose objects are open sets diffeomorphic to at most $k$ open balls in $\coprod_{i}P_i$.  

\begin{defin}
For a tuple of non-negative integers $\vec{\jmath}=(j_1,\ldots, j_m)$ and $\vec{U}=(U_1,\ldots, U_m)$ where $U_i\in\mathcal{O}(P_i)$, define $\mathcal{O}(\vec{U})=\prod_i\mathcal{O}(U_i)$, and $\mathcal{O}_{\vec{\jmath}}(\vec{U})=\prod_i\mathcal{O}_{j_i}(U_i)$.
\end{defin}



\subsection{Multivariable polynomial functors and multivariable Taylor tower}\label{S:MultiPolynomials}


This section is the analog of \refS{Polynomials}.  In fact, each of definitions and statements \ref{D:Poly}--\ref{P:iteratedT} from that section has a direct analog here. 
 In addition, this section contains a variety of examples of multivarible polynomial functors.

The equivalence of categories 
$$\mathcal{O}\left(\coprod_iP_i\right)\cong\prod_i\mathcal{O}(P_i)$$ gives two ways of viewing a functor $F$ defined on it. On the one hand, when we view the domain as the category $\mathcal{O}(\coprod_iP_i)$, we think of $F$ as a function of a single variable $U\in\mathcal{O}(\coprod_iP_i)$. Here we may apply the Goodwillie-Weiss single variable calculus from Section \ref{S:CalculusReview}. On the other hand, the category $\prod_i\mathcal{O}(P_i)=\mathcal{O}(\vec{P})$ has objects $\vec{U}=(U_1,\ldots, U_m)$ and we will develop a multivariable calculus which allows us to treat the variables separately. The first step is to introduce a suitable notion of a multivariable polynomial functor. It is clear what we should mean by a ``good'' functor. The analog of \refD{goodfunctor} is

\begin{defin}\label{D:multigoodfunctor}
A contravariant functor $F:\mathcal{O}(\vec{P})\longrightarrow\Spaces$ is \emph{good} if 
\begin{enumerate}
\item It takes isotopy equivalences to homotopy equivalences, and 
\item For any nested sequence $\{\vec{U}_i\}$ of open subsets of $\coprod_iP_i$, the map 
$$
F (\cup \vec{U}_i) \longrightarrow \underset{i}{\holim}\, \, F(\vec{U}_i)
$$
is a weak equivalence.
\end{enumerate}
\end{defin}

Now we will define what it means to be polynomial. Let $F:\mathcal{O}(\vec{P})\rightarrow\Spaces$ again be a good (contravariant) functor, let $\vec{U}=(U_1\ldots, U_m)\in\mathcal{O}(\vec{P})$, and let $\vec{\jmath}=(j_1,\ldots, j_m)$ be an $m$-tuple of non-negative integers. For all $1\leq i\leq m$, let $A_0^i,\ldots, A_{j_i}^i$ be pairwise disjoint closed subsets of $U_i$. For each $i$, let $U_{S_i}=U_i-\cup_{j\in S_i}A_{j}^i$, let $\vec{U}_{\vec{S}}=(U_{S_1},\ldots, U_{S_m})$, where $\vec{S}=(S_1,S_2,\ldots S_m)$ denote an element of $\mathcal{P}([j_1])\times\cdots\times\mathcal{P}([j_m])$. Recall that $\mathcal{P}([\vec{\jmath}])=\mathcal{P}([j_1])\times\cdots\times\mathcal{P}([j_m])$, and let $\mathcal{P}_0([\vec{\jmath}])=\mathcal{P}_0([j_1])\times\cdots\times\mathcal{P}_0([j_m])$ be the full subcategory of those $\vec{S}=(S_1,\ldots, S_m)$ such that $S_i\neq\emptyset$ for all $i$.

The following is the analog of \refD{Poly}.

\begin{defin}\label{D:Multipoly}
A good functor $F:\mathcal{O}(\vec{P})\rightarrow\Spaces$ is said to be \textsl{polynomial of degree $\leq\vec{\jmath}=(j_1,\ldots, j_m)$} if, for all $\vec{U}=(U_1,\ldots, U_m)\in\mathcal{O}(\vec{P})$ and for all pairwise disjoint closed subsets $A_0^i,\ldots, A_{j_i}^i$ in $U_i$, $1\leq i\leq m$, the map 
\begin{equation}\label{E:polydef}
F(\vec{U})\longrightarrow\underset{\vec{S}\in\mathcal{P}_0([\vec{\jmath}])}{\holim}\, \, F(\vec{U}_{\vec{S}})
\end{equation}
is an equivalence.
\end{defin}

\begin{rem}
If $j_i=-1$ for some $i$ and the product category is empty, the target of the map in \refE{polydef} is a point, and hence only the constant functor, all of whose values are contractible, satisfies the definition. Note that we do not require the sets $A^i_k$ to be nonempty. In the case, $m=1$, the definition is trivially satisfied in case some $A_k=\emptyset$, since  the $(j+1)$-cube in question is homotopy cartesian because we may view it as a map of identical $j$-cubes. In general, suppose $A^i_k=\emptyset$ for some $1\leq i\leq m$ and some $0\leq k\leq j_i$. Since $F(U_1,\ldots, U_{S_i},\ldots, U_m)\to F(U_1,\ldots, U_{S_i\cup k},\ldots, U_m)$ is an equivalence for all $S_i\in\mathcal{P}_0[j_i]$ which do not contain $k$, the homotopy limit of $F$ over $\mathcal{P}_0[j_i]$ is equivalent to the homotopy limit of $F$ over the full subcategory of those $S_i\in\mathcal{P}_0[j_i]$ such that $k\in S_i$. This has an initial object, namely $\{k\}$, and hence the homotopy limit over this category is equivalent to $F(U_1,\ldots, U_i,\ldots, U_m)$. It follows that the homotopy limit of $F$ over $\mathcal{P}_0[\vec{\jmath}]$ is equivalent to the homotopy limit over $\mathcal{P}_0[j_1]\times\cdots\times\mathcal{P}_0[j_{i-1}]\times\mathcal{P}_0[j_{i+1}]\times\cdots\times\mathcal{P}_0[j_m]$.
\end{rem}

\begin{prop}\label{P:multipolyseparate}
A good functor $F:\mathcal{O}(\vec{P})\rightarrow\Spaces$ is polynomial of degree $\leq\vec{\jmath}=(j_1,\ldots, j_m)$ if and only if it is polynomial of degree $j_i$ in the $i$th variable for all $i$.
\end{prop}

\begin{proof}
If $F$ is polynomial of degree $j_i$ in the $i$th variable for all $i$, then for all $\vec{U}=(U_1,\ldots, U_m)$ and for all pairwise disjoint closed subsets $A^i_0,\ldots, A^i_{j_i}$, the map 
$$
F(U_1,\ldots, U_i,\ldots U_m)\longrightarrow\underset{S_i\in\mathcal{P}_0[j_i]}{\holim}\,\, F(U_1,\ldots, U_{S_i},\ldots, U_m)
$$ 
is an equivalence. It follows that 
$$
F(U_1,\ldots U_m)\longrightarrow\underset{S_1\in\mathcal{P}_0[j_1]}{\holim}\cdots \underset{S_i\in\mathcal{P}_0[j_m]}{\holim}\,\, F(U_{S_1},\ldots, U_{S_m})
$$ 
is an equivalence, and if we put $\vec{\jmath}=(j_1,\ldots, j_m)$ and repeatedly apply \refP{holimproduct}, we see that
$$
F(\vec{U})\longrightarrow\underset{\vec{S}\in\mathcal{P}_0[\vec{\jmath}]}{\holim}\,\, F(\vec{U}_{\vec{S}})
$$ 
is an equivalence as well. For the converse, suppose $F$ is polynomial of degree $\leq\vec{\jmath}=(j_1,\ldots, j_m)$. Then, in \refD{Multipoly}, for $k\neq i$, choose $A^k_{j}$ to be empty for all $j$, and let $A_0^i,\ldots, A_{j_i}^i$ be pairwise disjoing closed subsets of $U_i$. Then by the remarks following \refD{Multipoly}, the map $F(U_1,\ldots, U_i,\ldots, U_m)\to\holim_{S_i\in\mathcal{P}_0[j_i]}F(U_1,\ldots, U_{S_i},\ldots, U_m)$ is an equivalence.
\end{proof}


Now we will explore some examples so the reader gets a feel for this definition. We will also compare the multi degree of our functors to their single variable degree in the following examples. We leave it to the reader to verify that these really are polynomial of the asserted degrees.


\begin{example}
The functor $(U,V)\mapsto\Map(U\coprod V,X)\simeq\Map(U,X)\times\Map(V,X)$ is polynomial of degrees $\leq(1,1)$. This is easy to deduce from the fact that $\Map(-,X)$ takes homotopy pushout squares to homotopy pullback squares. As a functor of the single variable $U\coprod V$ it is a polynomial of degree $\leq 1$. An analog for this functor is the function $f(x,y):\R^2\rightarrow\R$ given by $f(x,y)=x+y$.
\end{example}

\begin{example} Everything from the previous example applies equally well to the functor $(U,V)\mapsto\Imm(U\coprod V,X)\simeq\Imm(U,X)\times\Imm(V,X)$. That this functor takes homotopy pushout squares to homotopy pullback squares is the content of the Smale-Hirsch Theorem. See \cite[Example 2.3]{W:EI1} for details.
\end{example}

\begin{example}
The functor $(U,V)\mapsto\Map\left({U\choose j}\coprod {V\choose k},X\right)$ is polynomial of degrees $\leq (j,k)$. Once again these are all easy to deduce from the fact that $\Map(-,X)$ takes homotopy pushout squares to homotopy pullback squares. As a functor of $U\coprod V$, it is a polynomial of degree $\leq \max\{j,k\}$. Its analog is $f(x,y)=p(x)+q(y)$, where $p(x)$ is a polynomial of degree $j$ and $q(y)$ is a polynomial of degree $k$.
\end{example}

\begin{example}
The functor $(U,V)\mapsto\Map(U\times V,X)$, where $(U,V)\in\mathcal{O}(P)\times\mathcal{O}(Q)$, is polynomial of degrees $\leq (1,1)$. Again, these are easily verified. It is polynomial of degree $\leq 2$ as a functor on $\mathcal{O}(P\coprod Q)$ (but not of degree $\leq 1$). Its analog is $f(x,y)=xy+(\mbox{lower degree terms})$.
\end{example}

\begin{example}
The functor $(U,V)\mapsto\Map\left({U\choose j}\times {V\choose k},X\right)$, where $(U,V)\in\mathcal{O}(P)\times\mathcal{O}(Q)$, is polynomial of degrees $\leq (j,k)$. This functor is polynomial of degree $\leq j+k$ as a functor on $\mathcal{O}(P\coprod Q)$, and its analog is $f(x,y)=x^jy^k+(\mbox{lower degree terms})$.
\end{example}

\begin{example}
The functor $(U,V)\mapsto\Map(U,\Emb(V,N))$ is polynomial of degree $\leq 1$ in $U$, but (as a functor of $V$ alone), $\Emb(V,N)$ is not polynomial of degree $\leq k$ for any $k$.
\end{example}



We now have a definition analogous to \refD{kStage}.

\begin{defin}\label{D:jMultiStage}  Define the \emph{$\vec{\jmath}^{th}$ degree Taylor approximation of $F$} to be
$$T_{\vec{\jmath}}F(\vec{U})=\underset{\mathcal{O}_{\vec{\jmath}}(\vec{U})}{\holim}\, \, F.$$
\end{defin}

If $\vec{U}\in\mathcal{O}_{\vec{\jmath}}(\vec{U})$, then $\vec{U}$ is a final object in $\mathcal{O}_{\vec{\jmath}}(\vec{U})$, and since $F$ is contravariant, we have that $F(\vec{U})\rightarrow T_{\vec{\jmath}}F(\vec{U})$ is an equivalence.

Note that again one has canonical maps $F\to T_{\vec{\jmath}}F$ and $T_{\vec{\jmath}}F\to T_{\vec{\jmath'}}F$ for $\vec{\jmath'}\leq\vec{\jmath}$.  Recalling the definition of the poset $\mathcal{Z}^m$ from \refS{Conventions}, we thus get the following analog of \refD{SingleTaylorTower}.

\begin{defin}\label{D:kStageMulti}
The \emph{multivariable Taylor tower (or Taylor multi-tower) of $F$} is the diagram (contravariant functor)

$$
\vec{\jmath}   \longmapsto  T_{\vec{\jmath}}F
$$

Also let 
$$
T_{\vec{\infty}}F = \underset{\vec{\jmath}\in \mathcal{Z}^m}{\holim}\, T_{\vec{\jmath}}F.
$$
Polynomial approximation $T_kF$ will also be called  the \emph{$\vec{\jmath}^{th}$ stage} of the Taylor multi-tower of $F$.
\end{defin}

Of course, since $\mathcal{Z}^m\cong\prod_{i=1}^m\mathcal{Z}$ as posets, \refP{holimproduct} implies that 

$$T_{\vec{\infty}}F\cong \underset{j_1\in \mathcal{Z}}{\holim}\,\cdots\underset{j_m\in \mathcal{Z}}{\holim}\, T_{(j_1,\ldots, j_m)}F.$$

It is precisely this multi-tower, and its finite model from \refS{FiniteModels}, which will be of most interest to us in future work, in particular in \cite{MV:Links}. We have the following generalization of Theorem \ref{T:polyclass}.

\begin{thm}\label{T:Multipolyclass}
Let $F_1\rightarrow F_2$ be a natural transformation of good functors. Suppose that $F_1$ and $F_2$ are polynomials of degree $\leq \vec{\jmath}=(j_1,\ldots, j_m)$. If $F_1(\vec{U})\rightarrow F_2(\vec{U})$ is an equivalence for all $\vec{U}\in\mathcal{O}_{\vec{\jmath}}(\vec{P})$, then it is an equivalence for all $\vec{U}\in \mathcal{O}(\vec{P})$.
\end{thm}

\begin{proof}
This follows from repeated application of the proof of \refT{polyclass} in \cite{W:EI1} and \refP{multipolyseparate}. Let $\vec{U}=(U_1,\ldots, U_m)$. Since the $F_i$ are polynomial of degree $\leq(j_1,\ldots, j_m)$, it follows from the proof of \refT{polyclass} that $F_1(U_1,\ldots, U_m)\to F_2(U_1,\ldots, U_m)$ is an equivalence for all $U_1\in\mathcal{O}(P_1)$, and for all $U_i\in\mathcal{O}_{j_i}(\vec{P})$. By induction on $m$ and \refP{holimproduct}, the result follows.
\end{proof}

In analogy with \refC{polyiffTk}, we have the following corollary, whose proof we omit.

\begin{prop}\label{P:multipolyiffTk}
A good functor $F$ is polynomial of degree $\leq \vec{\jmath}$ if and only if the canonical map $F(\vec{U})\to T_{\vec{\jmath}}F(\vec{U})$ is an equivalence for all $\vec{U}\in\mathcal{O}(\vec{P})$.
\end{prop}

Now we have an analog of \refT{tkpoly}.

\begin{thm}\label{T:MultiTkpoly}\ \ 
\begin{enumerate}
\item $T_{\vec{\jmath}}F$ is polynomial of degree $\leq\vec{\jmath}$.
\item If $F$ is polynomial of degree $\leq\vec{\jmath}$, then $F\rightarrow T_{\vec{\jmath}}F$ is an equivalence.
\item If $F$ is polynomial of degree $\leq \vec{\jmath}$, then it is polynomial of degree $\leq \vec{k}$ for all $\vec{k}\geq\vec{\jmath}$.
\end{enumerate}
\end{thm}

These are straightforward adaptations of the proofs of the corresponding statements from \refT{tkpoly}. Note that (2) is equivalent to \refT{Multipolyclass}, which can be seen by considering the following commutative diagram.

$$\xymatrix{
F_1  \ar[r]\ar[d] & F_2 \ar[d] \\
T_{\vec{\jmath}} F_1\ar[r] &T_{\vec{\jmath}}F_2.\\
}
$$
On the one hand, by \refT{MultiTkpoly} (2), If each of the $F_i$ are polynomial of degree $\leq \vec{\jmath}$, then the vertical arrows are equivalences. If the natural transformation $F_1\to F_2$ is an equivalence for all $\vec{U}\in\mathcal{O}_{\vec{\jmath}}(\vec{P})$, then the bottom horizontal arrow is an equivalence because the homotopy limit in question is taken over all such $\vec{U}$, and homotopy limits preserve weak equivalences. Hence the top horizontal arrow is an equivalence for all $\vec{V}\in\mathcal{O}(\vec{P})$. On the other hand, \refT{Multipolyclass} implies the vertical arrows in the above diagram are equivalences immediately, because $F\to T_{\vec{\jmath}}F$ is a natural transformation between polynomials of degree $\leq \vec{\jmath}$ which is an equivalence for all $\vec{U}\in\mathcal{O}_{\vec{\jmath}}(\vec{P})$, which is precisely what \refT{MultiTkpoly} claims.

In analogy to Proposition \ref{P:factorthru}, we have the following result.

\begin{prop}
Let $F,G:\mathcal{O}(\vec{P})\rightarrow\Spaces$ be good functors. Suppose $G$ is a polynomial of degree $\leq\vec{\jmath}$. Then a natural transformation $F\rightarrow G$ factors through $T_{\vec{\jmath}}F$ up to some form of weak equivalence.
\end{prop}

\begin{proof}
Since $G$ is polynomial of degree $\leq\vec{\jmath}$, there is a homotopy equivalence $G\rightarrow T_{\vec{\jmath}}G$. A natural transformation $F\rightarrow G$ gives rise to a natural transformation $T_{\vec{\jmath}}F\rightarrow T_{\vec{\jmath}}G$, and, using a homotopy inverse of $G\rightarrow T_{\vec{\jmath}}G$, we obtain the desired result.
\end{proof}

Finally, we also have an analog of \refP{iteratedT}.  This will be used in the proof of \refP{multilayerhomog}.

\begin{prop}\label{P:iteratedvecT}
There is an equivalence of functors $$T_{\vec{\jmath}}T_{\vec{k}}F\simeq T_{\min\{\vec{\jmath},\vec{k}\}}F\simeq T_{\vec{k}}T_{\vec{\jmath}}F,$$ where $\min\{\vec{\jmath},\vec{k}\}=(\min\{j_1,k_1\},\min\{j_2,k_2\},\ldots, \min\{j_m,k_m\})$.
\end{prop}

\begin{proof}
Let $\vec{U}=(U_1,\ldots, U_m)\in\mathcal{O}$.

By definition,
$$
T_{\vec{\jmath}}T_{\vec{k}}F(\vec{U})=\underset{\vec{V}\in\mathcal{O}_{\vec{\jmath}}(\vec{U})}{\holim}\, \, 
T_{\vec{k}}F(\vec{V}).
$$
By \refP{holimproduct}, we may write
$$
\underset{\vec{V}\in\mathcal{O}_{\vec{\jmath}}(\vec{U})}{\holim}\, \, 
T_{\vec{k}}F(\vec{V})\cong\underset{V_1\in\mathcal{O}_{j_1}(U_1)}{\holim}\,  \cdots \underset{V_m\in\mathcal{O}_{j_m}(U_m)}{\holim}\, 
T_{\vec{k}}F(V_1,\ldots, V_m).
$$
Applying this result again allows us to rewrite the above as
$$
\underset{V_1\in\mathcal{O}_{j_1}(U_1)}{\holim}\,  \cdots \underset{V_m\in\mathcal{O}_{j_m}(U_m)}{\holim}\, \underset{W_1\in\mathcal{O}_{k_1}(V_1)}{\holim}\,  \cdots \underset{W_m\in\mathcal{O}_{k_m}(V_m)}{\holim}\, 
F(W_1,\ldots, W_m).
$$

Now use the fact that homotopy limits commute and \refP{iteratedT} to arrive at the desired result.
\end{proof}

\subsection{Convergence of the multivariable Taylor tower for the embedding functor}\label{S:MultiConvergence}

Recall that Theorem \ref{T:gwemb} says that when $n-m-2>0$, the natural map 
$$\Emb(M,N)\rightarrow\underset{k}{\holim}\, T_k\Emb(M,N)$$ is an equivalence.  Here is an immediate generalization.



\begin{thm}\label{T:MultiConvergence}
Let $M=\coprod_{i=1}^{m}P_i$, $\vec{P}=(P_1,\ldots, P_m)$, where $\dim(P_i)=p_i$ and suppose $n-p_i-2>0$ for all $i$.  Then the map 
$$\Emb(\vec{P},N)\longrightarrow T_{\vec{\infty}}\Emb(\vec{P},N)$$
is an equivalence.
\end{thm}

\begin{proof}
First let us fix some notation, since we will be looking at the single variable Taylor approximations in each variable separately. For a functor $F:\mathcal{O}(\vec{P})\rightarrow\Spaces$ and $\vec{U}=(U_1,\ldots, U_m)\in\mathcal{O}(\vec{P})$, let $\del^{i} T_{j_i}F(\vec{U})=\holim_{\mathcal{O}_{j_i}(U_{i})}F(\vec{U})$. We will consider the case $m=2$; the general result will follow by induction.

For $M=P_1\coprod P_2$, we have a fibration sequence 
 $$\Emb(P_1,N-P_2)\longrightarrow\Emb(P_1\coprod P_2,N)\longrightarrow\Emb(P_2,N).$$

Applying $\del^2T_{j_2}$ everywhere yields a fibration sequence

$$\del^2T_{j_2}\Emb(P_1,N-P_2)\longrightarrow\del^2T_{j_2}\Emb(P_1\coprod P_2,N)\longrightarrow\del^2T_{j_2}\Emb(P_2,N).$$

The map $$\Emb(P_2,N)\longrightarrow\del^2T_{j_2}\Emb(P_2,N)$$ is $(j_2(n-p_2-2)+1-p_2)$-connected by \refT{gwemb}. The map $$\Emb(P_1,N-P_2)\longrightarrow\del^2T_{j_2}\Emb(P_1,N-P_2)$$ is $((j_2+1)(n-p_2-2)+1-p_1)$-connected by work of Goodwillie-Klein \cite{GK}; indeed, such connectivity estimates are part of what go into the proof of \refT{gwemb}. In particular, as functors of $P_2$, the towers for $\Emb(P_2,N)$ and $\Emb(P_1,N-P_2)$ converge since $n-p_2-2>0$. Thus, taking homotopy limits over $j_2$ gives a fibration sequence

$$\Emb(P_1,N-P_2)\longrightarrow\underset{j_2}{\holim}\;\del^2T_{j_2}\Emb(P_1\coprod P_2,N)\longrightarrow\Emb(P_2,N).$$

Now apply $\del^1T_{j_1}$ everywhere to obtain another fibration sequence

$$\del^1T_{j_1}\Emb(P_1,N-P_2)\longrightarrow\del^1T_{j_1}\underset{j_2}{\holim}\;\del^2T_{j_2}\Emb(P_1\coprod P_2,N)\longrightarrow\del^1T_{j_1}\Emb(P_2,N).$$

The functor $\Emb(P_2,N)$ does not depend on $P_1$, so $\del^1T_{j_1}\Emb(P_2,N)\simeq\Emb(P_2,N)$. Once again \refT{gwemb} gives us a $(j_1(n-p_1-2)+1-p_1)$-connected map $\Emb(P_1,N-P_2)\rightarrow \del^1T_{j_1}\Emb(P_2,N-P_2)$, and thus an equivalence $$\Emb(P_1,N-P_2)\longrightarrow\holim_{j_1}\del^1T_{j_1}\Emb(P_1,N-P_2)$$ since $n-p_1-2>0$. Putting this all together gives a fibration sequence

$$\Emb(P_1,N-P_2)\longrightarrow\underset{j_1}{\holim}\;\del^1T_{j_1}\underset{j_2}{\holim}\;\del^2T_{j_2}\Emb(P_1\coprod P_2,N)\longrightarrow\Emb(P_2,N).$$

Thus by \refP{holimproduct}, and noting that $\del^1T_{j_1}\del^2T_{j_2}=T_{(j_1,j_2)}$ (again by \refP{holimproduct}), we have an equivalence

$$\Emb(P_1\coprod P_2,N)\rightarrow\underset{(j_1,j_2)}{\holim}\;T_{(j_1,j_2)}\Emb(P_1\coprod P_2,N).$$

The general result follows by induction on $m$.

\end{proof}

\begin{rem}
\refT{MultiConvergence} follows from \refT{twotowers}, but this proof gives a slightly stronger result. In particular, we produced an equivalence $$\Emb(P_1\coprod P_2,N)\longrightarrow\del^2T_{\infty}\Emb(P_1\coprod P_2,N).$$ That is, the tower converges in each variable separately.
\end{rem}

\section{Multivariable homogeneous functors}\label{S:multihomog}


Recall that \refT{homogclass} classifies homogeneous functors in the single variable setting.  The goal of this section is to generalize this result to the multivariable case; this is done in \refT{multihomogclass}. We will begin by letting \refT{homogclass} tell us what it can about a possible classification of multivariable homogeneous functors. The main result of this analysis is \refP{homogdisjoint} in \refS{multihomogdef} as well as making sense of what goes into defining the codomain of this equivalence. \refS{multihomogdef} is both meant to motivate our definition of a homogeneous functor in the multivariable setting (\refD{MultiHomogeneous}) and to help set up for the statement and proof of the relationship between the single variable and multivariable towers in \refT{twotowers}. After defining homogeneous functors, we begin by exploring examples using spaces of sections in \refS{Sections}. These are important because every homogeneous functor can be constructed from a space-of-sections functor; this is the content of our Classification \refT{multihomogclass}. We will use the classification of homogeneous functors to prove \refT{twotowers} which relates the single and multivariable towers. 

Before we embark on our study  of homogeneous functors, it may help the reader to consider the following example from ordinary calculus.

\begin{example}\label{Ex:MultiPoly}
For an ordinary polynomial $p(x)=a_0+a_1x+\cdots+a_mx^m$, the $n^{th}$ degree approximation is just the truncation $p_n(x)=a_0+a_1x+\cdots+a_nx^n$, and the homogeneous degree $n$ part is $a_nx^n$. The degree $n-1$ approximation is again the truncation $p_{n-1}(x)=a_0+a_1x+\cdots+a_{n-1}x^{n-1}$, and we can think of obtaining the homogeneous degree $n$ part by taking the difference $h_n(x)=a_nx^n=p_n(x)-p_{n-1}(x)$. The analog of difference for us is of course the homotopy fiber. For a polynomial in two variables, 
$$
p_{n,k}(x,y)=a_{0,0}+a_{1,0}x+a_{0,1}y+\cdots + a_{n,k}x^ny^k,
$$ 
(where the subscript $(n,k)$ indicates that this is the term of highest bidegree), the homogeneous bidegree $(n,k)$ part is $a_{n,k}x^ny^k$. This can be obtained by subtracting all of the truncated polynomials $p_{i,j}(x,y)$, where $(i,j)=(n-1,k),(n,k-1),(n-1,k-1)$ from $p_{n,k}(x,y)$ ``with signs''. The formula for obtaining the homogeneous bidegree $(n,k)$ part of $p_{n,k}(x,y)$ is

$$h_{n,k}(x,y)=p_{n,k}(x,y)-p_{n-1,k}(x,y)-p_{n,k-1}(x,y)+p_{n-1,k-1}(x,y).$$


More generally, for a polynomial in $m$ variables of degree $\vec{\jmath}=(j_1,\ldots, j_m)$, $p_{\vec{\jmath}}=a_{\vec{0}}+\cdots+a_{\vec{\jmath}}\vec{x}^{\vec{\jmath}}$ (here $\vec{x}^{\vec{\jmath}}=x_1^{j_1}\cdots x_m^{j_m})$, we can obtain the homogeneous degree $\vec{\jmath}$ part with a similar iterated difference as follows. For a subset $R\subset\underline{m}$, let $\vec{\jmath}_R=(j_{1,R},\ldots, j_{m,R})$ be defined by

\begin{equation*}
j_{i,R}= 
\begin{cases}
j_i &\text{if $ i\notin R$;}\\
j_i-1, &\text{if $i\in R$.}
\end{cases}
\end{equation*}

Then the homogeneous multidegree $\vec{\jmath}$ part of $p_{\vec{\jmath}}(\vec{x})$ is 

$$h_{\vec{\jmath}}(\vec{x})=a_{\vec{\jmath}}x^{\vec{\jmath}}=p_{\vec{\jmath}}(\vec{x})-\sum_{R\neq\emptyset}(-1)^{|R|-1}p_{\vec{\jmath}_R}(\vec{x}).$$

The analog of such an expression is the total homotopy fiber, which can be thought of as an iterated homotopy fiber, or an ``iterated difference''. Our \refP{mcubehomog} will exhibit multivariable homogeneous functors as just such a total homotopy fiber.
\end{example}


\subsection{Extracting multivariable information from the single variable tower}\label{S:multihomogdef}


We will begin by seeing what the single variable classification tells us in the case where $M=\coprod_iP_i$. Let $F\colon \mathcal{O}(M)\rightarrow\Spaces$ be a good functor, and choose a basepoint for $F(M)$. Theorem \ref{T:homogclass} says that the functor $$L_kF=\hofiber(T_kF\longrightarrow T_{k-1}F)$$ maps by an equivalence to the space of compactly supported sections of a fibration $Z\rightarrow {M\choose k}$. The following is an easy consequence of that statement, and will help motivate our definition of a homogeneous functor.



\begin{prop}\label{P:homogdisjoint}
If $M=\coprod_{i=1}^mP_i$, then the section spaces of the classifying fibration for $L_kF$ break up into a product of section spaces. That is, 

$$L_kF\simeq\Gamma^c\left({U\choose k},Z;p\right)=\prod_{|\vec{\jmath}|=k}\Gamma^c\left({\vec{U}\choose \vec{\jmath}},Z;p\right),$$

where $\vec{U}=(U_1,\ldots, U_m)$ and $U_i=U\cap P_i$.
\end{prop}

\begin{proof}
When $M=\coprod_{i=1}^mP_i$, we have that $${M\choose k}=\coprod_{\abs{\vec{\jmath}}=k}{\vec{P}\choose \vec{\jmath}}.$$  The result then follows by noting that a section of a fibration $p:Z\rightarrow X\coprod Y$ is a pair of sections, one defined over $X$ and the other over $Y$. That is, the functor $\Gamma(-,Z;p)$ turns coproducts into products.
\end{proof}

The reader may protest that we have not defined $\Gamma^c\left({\vec{U}\choose \vec{\jmath}},Z;p\right)$. This is \refD{homogdegreej}, and we will devote the remainder of this section to explaining what goes into this definition. Recall that

$$\Gamma^c\left({U\choose k},Z;p\right)=\hofiber\left(\Gamma\left({U\choose k},Z;p\right)\rightarrow\Gamma\left(\del{U\choose k},Z;p\right)\right),$$

and recall from \refE{bdysections} that

\begin{equation}\label{E:bdysections2}
\Gamma\left(\del{U\choose k},Z;p\right)=\underset{Q\in\mathcal{N}}{\hocolim}\;\Gamma\left({U\choose k}\cap Q,Z;p\right).
\end{equation}

Also recall that $\mathcal{N}$ is the poset of neighborhoods $Q$ of $\Delta_{fat}$ in $sp_kU$. There is an isomorphism of categories $\mathcal{N}\cong\prod_{|\vec{\jmath}|=k}\mathcal{N}_{\vec{\jmath}}$, where $\mathcal{N}_{\vec{\jmath}}$ is the poset of neighborhoods $Q_{\vec{\jmath}}$ of $\Delta_{\vec{fat}}$ in $sp_{\vec{\jmath}}\, \vec{P}$. Every element $Q\in\mathcal{N}_{\vec{\jmath}}$ can be uniquely expressed as a disjoint union $\coprod_{|\vec{\jmath}|=k}Q_{\vec{\jmath}}$, where $Q_{\vec{\jmath}}\in\mathcal{N}_{\vec{\jmath}}$.

Hence, expanding upon \refE{bdysections2}, writing $\vec{Q}$ for the element of $\prod_{|\vec{\jmath}|=k}\mathcal{N}_{\vec{\jmath}}$ corresponding to $Q=\coprod_{|\vec{\jmath}|=k}Q_{\vec{\jmath}}\in\mathcal{N}$, and with $U=\coprod_i U_i$, we have

\begin{eqnarray*}
\underset{Q\in\mathcal{N}}{\hocolim}\;\Gamma\left({U\choose k}\cap Q,Z;p\right)&\simeq&\underset{\vec{Q}\in\prod_{|\vec{\jmath}|=k}\mathcal{N}_{\vec{\jmath}}}{\hocolim}\;\Gamma\left({U\choose k}\cap \coprod_{|\vec{\jmath}|=k}Q_{\vec{\jmath}},Z;p\right)\\
&\simeq& \underset{\vec{Q}\in\prod_{|\vec{\jmath}|=k}\mathcal{N}_{\vec{\jmath}}}{\hocolim}\;\underset{|\vec{\jmath}|=k}{\prod}\Gamma\left({\vec{U}\choose \vec{\jmath}}\cap Q_{\vec{\jmath}},Z;p\right)\\
&\simeq&\underset{|\vec{\jmath}|=k}{\prod}\underset{Q_{\vec{\jmath}\in\mathcal{N}_{\vec{\jmath}}}}{\hocolim}\,\Gamma\left({\vec{U}\choose \vec{\jmath}}\cap Q_{\vec{\jmath}},Z;p\right).\\
\end{eqnarray*}

The second equivalence comes from the fact that $\Gamma(-)$ takes coproducts to products (more generally, homotopy colimits to homotopy limits), and that ${U\choose k}=\coprod_{|\vec{\jmath}|=k}{\vec{U}\choose\vec{\jmath}}$. The last equivalence uses two facts. The first is that a homotopy colimit over a product of categories can be written as an iterated homotopy colimit, a fact analogous to \refP{holimproduct}. The second is that finite homotopy limits commute with filtered homotopy colimits, and in this case, each homotopy colimit in our iterated series is filtered. See \cite[Sublemma 7.4]{W:EI1} for details.

\begin{equation}\label{E:multisections1}
\Gamma^c\left({\vec{U}\choose \vec{\jmath}},Z;p\right)=\hofiber\left(\Gamma\left({\vec{U}\choose \vec{\jmath}},Z;p\right)\rightarrow\underset{Q_{\vec{\jmath}\in\mathcal{N}_{\vec{\jmath}}}}{\hocolim}\,\Gamma\left({\vec{U}\choose {\vec{\jmath}}}\cap Q_{\vec{\jmath}},Z;p\right)\right).
\end{equation}

We can take this analysis a bit further and continue to decompose the homotopy colimit which appears in \refE{multisections1}. In order to do this, we need to recall the decomposition $\Delta_{\vec{fat}}(sp_{\vec{\jmath}}\,\vec{U})=\cup_{i=1}^m\Delta_{fat_i}(sp_{\vec{\jmath}}\,\vec{U})$. Also recall that 

$$\Delta_{fat_i}(sp_{\vec{\jmath}}\,\vec{U})=sp_{j_1}\,U_1\times\cdots\times sp_{j_{i-1}}\,U_{i-1}\times \Delta_{fat}(sp_{j_{i}}\,U_{i})\times sp_{j_{i+1}}\,U_{i+1}\times\cdots\times sp_{j_m}\,U_m.$$

For the remainder of this section we will work with a fixed $\vec{\jmath}$ to avoid cumbersome notation whenever possible. The discussion about fat diagonals just above motivates the following definition.

\begin{defin}\label{D:nbhdfat}
Let $R\subset\underline{m}$ and let $N_i$ be an open neighborhood of $\Delta_{fat}$ in $sp_{j_i}P_i$. Define a poset $\mathcal{N}_R$ whose objects are neighborhoods $Q=Q_1\times\cdots\times Q_m$, where 
\begin{equation*}
Q_i= 
\begin{cases}
{P_i\choose j_i}, &\text{if $ i\notin R$;}\\
N_i\cap{P_i\choose j_i}, &\text{if $i\in R$.}
\end{cases}
\end{equation*}
\end{defin}

Each $Q_{\vec{\jmath}}$ can be uniquely expressed as a union $Q_{\vec{\jmath}}=\cup_{i=1}^mQ_i$, where $Q_i\in\mathcal{N}_{\{i\}}$, and thus we have an isomorphism of categories $\mathcal{N}_{\vec{\jmath}}\cong\prod_{i=1}^m\mathcal{N}_{\{i\}}$.


The following proposition rewrites the target of the map in \refE{multisections1}.

\begin{prop}\label{P:multibdysections}
We have an equivalence

$$\underset{Q_{\vec{\jmath}\in\mathcal{N}_{\vec{\jmath}}}}{\hocolim}\,\Gamma\left({\vec{U}\choose {\vec{\jmath}}}\cap Q_{\vec{\jmath}},Z;p\right)\simeq\underset{R\neq\emptyset}{\holim}\,\underset{Q_R\in\mathcal{N}_R}{\hocolim}\,\Gamma\left({\vec{U}\choose {\vec{\jmath}}}\cap Q_R,Z;p\right).$$

\end{prop}

\begin{proof}
For $R\subset\underline{m}$, put $Q_R=\cap_{i\in R}Q_i$, where $Q_i\in\mathcal{N}_i$. Of course, $Q_R\in\mathcal{N}_R$. We have the following sequence of equivalences.

\begin{eqnarray*}
\underset{Q_{\vec{\jmath}\in\mathcal{N}_{\vec{\jmath}}}}{\hocolim}\,\Gamma\left({\vec{U}\choose {\vec{\jmath}}}\cap Q_{\vec{\jmath}},Z;p\right)&\simeq&\underset{(Q_1,\ldots, Q_m)\in\prod_{i=1}^m\mathcal{N}_{i}}{\hocolim}\;\Gamma\left({U\choose k}\cap \hocolim_{R\neq\emptyset}Q_{R},Z;p\right)\\
&\simeq& \underset{(Q_1,\ldots, Q_m)\in\prod_{i=1}^m\mathcal{N}_{i}}{\hocolim}\;\underset{R\neq\emptyset}{\holim}\;\Gamma\left({U\choose k}\cap Q_{R},Z;p\right)\\
&\simeq&\underset{R\neq\emptyset}{\holim}\;\underset{Q_r\in\mathcal{N}_{R}}{\hocolim}\;\Gamma\left({U\choose k}\cap Q_{R},Z;p\right).\\
\end{eqnarray*}

The second to last equivalence follows from the fact that $\Gamma(-)$ sends homotopy colimits to homotopy limits. The last equivalence follows from the fact that a homotopy colimit over a product of categories can be written as an iterated homotopy colimit of filtered categories, and filtered homotopy colimits commute with finite homotopy limits. Again, \cite[Sublemma 7.4]{W:EI1} applies.
\end{proof}

Now that we understand that the source and target of the map in \refE{multisections1} are products and the map between them is a product of maps, it is sensible to make the following definitions.

\begin{defin}\label{D:homogdegreej}
Define $$\Gamma\left(\del_R{\vec{U}\choose\vec{\jmath}}\right)=\underset{Q\in\mathcal{N}_R}{\hocolim}\, \Gamma\left({\vec{U}\choose \vec{\jmath}}\cap Q,Z;p\right),$$ and define $$\Gamma^c\left({\vec{U}\choose \vec{\jmath}},Z;p\right)=\hofiber\left(\Gamma\left({\vec{U}\choose \vec{\jmath}},Z;p\right)\rightarrow\underset{R\neq\emptyset}{\holim}\,\Gamma\left(\del_R{\vec{U}\choose \vec{\jmath}},Z;p\right)\right).$$
\end{defin}





Thus a good candidate for the degree $\vec{\jmath}=(j_1,\ldots, j_m)$ homogeneous layer of the Taylor tower for $F$ is the space of compactly supported sections $\Gamma^c\left({\vec{U}\choose \vec{\jmath}},Z;p\right)$. The following is a definition of homogeneous of degree $\vec{\jmath}$ which makes this true.

\begin{defin}\label{D:MultiHomogeneous}
A good functor $E:\mathcal{O}(\vec{P})\rightarrow\Spaces$ is \emph{homogeneous of degree $\vec{\jmath}$} if it is polynomial of degree $\leq \vec{\jmath}$ and $\holim_{\vec{k}<\vec{\jmath}}T_{\vec{k}}E(\vec{U})$ is contractible for all $\vec{U}$.
\end{defin}

\begin{rem}
It may appear strange that the homotopy limit over $R$ is not a part of this definition; after all, we just got through guessing what the candidate for the homogeneous layers of a good functor should look like, and the answer involved such a homotopy limit. We refer the reader to \refP{mcubehomog} to clear this up. The definition we have given is easier to work with in many instances. 
The other aspect that may appear puzzling to the reader is requiring $\holim_{\vec{k}<\vec{\jmath}}T_{\vec{k}}E(\vec{U})$ to be contractible instead of the stronger condition that $T_{\vec{k}}E(\vec{U})$ is contractible for all $\vec{k}<\vec{\jmath}$. It will follow from the classification of homogeneous functors that if $E$ is homogeneous of degree $\vec{\jmath}$, then $T_{\vec{k}}E(\vec{U})$ is contractible for all $\vec{U}$ and for all $\vec{k}<\vec{\jmath}$. See \refC{homogalt} for a full statement and explanation.
\end{rem}

For the remainder of this section, choose a basepoint in $F(\vec{P})$.


\begin{defin}\label{D:MultiLayer}
We define the \emph{$\vec{\jmath}^{th}$ layer of the Taylor multi-tower} of $F$ to be the functor

$$L_{\vec{\jmath}}F=\hofiber(T_{\vec{\jmath}}F\rightarrow \underset{\vec{k}<\vec{\jmath}}{\holim}\, T_{\vec{k}}F).$$

\end{defin}

\begin{rem}\label{R:MultiLayerCube} As mentioned earlier, 
we will see in \refP{mcubehomog} below that we can think of $L_{\vec{\jmath}}F$ as the total homotopy fiber of an $m$-cube of functors.
\end{rem}

\begin{prop}\label{P:multilayerhomog}
$L_{\vec{\jmath}}F$ is homogeneous of degree $\vec{\jmath}$.
\end{prop}

\begin{proof}
It is clear that $L_{\vec{\jmath}}F$ is polynomial of degree $\leq \vec{\jmath}$, since $T_{\vec{\jmath}}F$ is and so are $T_{\vec{k}}F$ for $\vec{k}<\vec{\jmath}$. To check the second condition, consider the following sequence of equivalences.

\begin{align*}
\underset{\vec{l}<\vec{\jmath}}{\holim}\, T_{\vec{l}}L_{\vec{\jmath}}F =&\, \underset{\vec{l}<\vec{\jmath}}{\holim}\, T_{\vec{l}}\hofiber(T_{\vec{\jmath}}F\rightarrow \underset{\vec{k}<\vec{\jmath}}{\holim}\, T_{\vec{k}}F)\\
 \simeq &\,  \hofiber(\underset{\vec{l}<\vec{\jmath}}{\holim}\, T_{\vec{l}}F\rightarrow\underset{\vec{l}<\vec{\jmath}}{\holim}\, \underset{\vec{k}<\vec{\jmath}}{\holim}\, T_{\min\{\vec{l},\vec{k}\}}F)\\
 \simeq & \, \hofiber(\underset{\vec{l}<\vec{\jmath}}{\holim}\, T_{\vec{l}}F\rightarrow\underset{\vec{l}<\vec{\jmath}}{\holim}\, T_{\vec{l}}F)\\
 \simeq &\, \ast
\end{align*}

The second and penultimate equivalences follow from \refP{iteratedvecT} and the fact that homotopy limits and homotopy fibers commute.
\end{proof}


\subsection{Spaces of sections}\label{S:Sections}


We devote this section to studying a particular class of homogeneous functors, namely those which arise as the space of sections of a fibration. As in the single variable case, these are important examples, because in the sense of \refT{multihomogclass} below, they are universal.

Let $\vec{P}=(P_1,\ldots, P_m)$,  $\vec{\jmath}=(j_1,\ldots, j_m)$, and  
$$p\colon Z\longrightarrow {\vec{P}\choose \vec{\jmath}}=\prod_i{P_i\choose j_i}$$ be a fibration.

\begin{defin}
Define $$\vec{U}\longmapsto\Gamma\left({\vec{U}\choose \vec{\jmath}},Z;p\right)$$ to be the functor which associates to $\vec{U}$ the space of sections of the fibration $p$. When  $p$ is understood, we may write this functor as $\Gamma{\vec{U}\choose\vec{\jmath}}$ for short.
\end{defin}

\begin{lemma}\label{L:sectionspoly}
$\Gamma{\vec{U}\choose\vec{\jmath}}$ is polynomial of degree $\leq \vec{\jmath}$.
\end{lemma}

\begin{proof}
Let $\vec{U}=(U_1,\ldots, U_m)\in\mathcal{O}(\vec{P})$, and for each $i$, let $A^i_0,\ldots, A^i_{j_i}$ be pairwise disjoint closed subsets of $U_i$. Put $U_{S_i}=U_i-\cup_{a\in S_i} A^i_a$, where $S_i\subset\{0,1,\ldots, j_i\}$. The $\vec{U}_{\vec{S}}\choose\vec{\jmath}$ cover $\vec{U}\choose\vec{\jmath}$ by the pigeonhole principle, and so the map 
$$
\underset{\vec{S}\in\mathcal{P}_0([\vec{\jmath}])}{\hocolim}{\vec{U}_{\vec{S}}\choose \vec{\jmath}}\longrightarrow {\vec{U}\choose\vec{\jmath}}
$$ 
is an equivalence. Applying $\Gamma(-)$ turns the homotopy colimit into a homotopy limit and gives the desired equivalence.
\end{proof}

The remainder of this section is concerned with building new functors from $\Gamma{\vec{U}\choose \vec{\jmath}}$, namely those which should be thought of as spaces of germs of sections near the fat diagonal. We also explore the relationship between these functors and the approximations $T_{\vec{k}}\Gamma{\vec{U}\choose \vec{\jmath}}$ for $\vec{k}<\vec{\jmath}$. \refL{tjrlemma} will tell us that these are lower-degree approximations to $\Gamma{\vec{U}\choose \vec{\jmath}}$, and this will aid us in building a homogeneous degree $\vec{\jmath}$ functor from it.

Our immediate goal is to prove \refL{germsofsectionspoly} below, and to do so we need to understand the homotopy colimit which appears in \refD{homogdegreej} a bit better. Following \cite{W:EI1}, choose a smooth triangulation on $\prod_i U^{j_i}_i$ which is equivariant with respect to the action of $\Sigma_{j_1}\times\cdots\times\Sigma_{j_m}$. This endows $sp_{\vec{\jmath}}\vec{U}=\prod_i U_i^{j_i}/\Sigma_{j_i}$ with a preferred PL structure with $\Delta_{\vec{fat}}$ as a PL subspace, and now we can talk about regular neighborhoods of $\Delta_{\vec{fat}}$.


Let $\mathcal{R}_R\subset \mathcal{N}_R$ denote the subposet of regular neighborhoods.

\begin{lemma}\label{L:regnbhdcofinal}
The inclusion $\mathcal{R}_R\rightarrow\mathcal{N}_R$ is left cofinal.
\end{lemma}

\begin{proof}
Every neighborhood contains a regular neighborhood.
\end{proof}

Since every regular neighborhood has the same homotopy type, the following lemma follows from \refL{regnbhdcofinal}. It will be useful in the proof of \refL{germsofsectionspoly}.


\begin{lemma}\label{L:regnbhd}
If $L$ is a regular neighborhood of $\Delta_{\vec{fat}}$, then the inclusion
$$\Gamma\left(\del_R{\vec{U}\choose\vec{\jmath}}\cap L,Z;p\right)\longrightarrow\underset{Q\in\mathcal{N}_R}{\hocolim}\, \Gamma\left({\vec{U}\choose \vec{\jmath}}\cap Q,Z;p\right)$$
is a homotopy equivalence for all $\vec{U}$.
\end{lemma}

This observation will afford us the freedom to choose $L$ to suit our needs in the next proof.

\begin{defin}\label{D:jsubr}
For $R\subset\underline{m}$ and $\vec{\jmath}=(j_1,\ldots, j_m)$, define $\vec{\jmath}_R=(j_{1,R},\ldots, j_{m,R})$ by
\begin{equation*}
j_{i,R}= 
\begin{cases}
j_i, &\text{if $ i\notin R$;}\\
j_i-1, &\text{if $i\in R$.}
\end{cases}
\end{equation*}
\end{defin}

We now have the following useful lemma.
\begin{lemma}\label{L:germsofsectionspoly}
$\Gamma\left(\del_R{\vec{U}\choose\vec{\jmath}}\right)$ is polynomial of degree $\leq\vec{\jmath}_R$.
\end{lemma}

\begin{proof}
This follows from \cite[Proposition 7.5]{W:EI1} and \refP{multipolyseparate}, since to check a functor is polynomial of a certain multi-degree, it is enough to check it one component at a time.

\end{proof}

Since $\Gamma\left(\del_R{\vec{U}\choose\vec{\jmath}}\right)$ is of degree $\leq\vec{\jmath}_R$, one might suspect that it is the degree $\vec{\jmath}_R$ approximation to $\Gamma{\vec{U}\choose\vec{\jmath}}$. This is indeed the case, according to the next lemma.

\begin{lemma}\label{L:tjrlemma}
The map $T_{\vec{\jmath}_R}\Gamma{\vec{U}\choose\vec{\jmath}}\rightarrow T_{\vec{\jmath}_R}\Gamma\left(\del_R{\vec{U}\choose\vec{\jmath}}\right)$ is an equivalence.
\end{lemma}

\begin{proof}
This is a straightforward adaptation of \cite[Proposition 7.6]{W:EI1}.
\end{proof}

We are almost ready to classify homogeneous functors, which is suggested by \refD{homogdegreej}. It tells us how to think about \refL{germsofsectionspoly} and \refL{tjrlemma}, and hints that functors which are homogeneous of degree $\vec{\jmath}$ should be of the form given in \refD{homogdegreej}.

From the discussion so far, we have most of what we need to conclude the following proposition. What we do not yet have is \refP{mcubehomog}, whose content fits better in the next section and is logically independent of what was discussed in this section.

\begin{prop}\label{P:sectionshomog}  The functor
$$\vec{U}\mapsto\Gamma^c\left(\del_R{\vec{U}\choose\vec{\jmath}}\right)$$ is homogeneous of degree $\vec{\jmath}$.
\end{prop}

\begin{proof}
This follows immediately from \refL{sectionspoly}, \refL{tjrlemma}, and \refP{mcubehomog}.
\end{proof}


\subsection{Classification of homogeneous functors}\label{S:Classification}


In this section we will prove the following analog of \refT{homogclass}.

\begin{thm}\label{T:multihomogclass}
If $E$ is homogeneous of degree $\leq \vec{\jmath}$, then there is an equivalence, natural in $\vec{U}\in\mathcal{O}(\vec{P})$,

$$E(\vec{U})\longrightarrow\hofiber\left(\Gamma\left({\vec{U}\choose\vec{\jmath}},Z;p\right)\longrightarrow\underset{R\neq\emptyset}{\holim}\, \Gamma\left(\del_R{\vec{U}\choose\vec{\jmath}},Z;p\right)\right).$$

Here $p:Z\rightarrow{\vec{P}\choose\vec{\jmath}}$ is some fibration with a preferred section defined near the fat diagonal.
\end{thm}

We will call $p$ the \emph{classifying fibration} for $E$.


As promised in the remark following \refD{MultiHomogeneous}, we have the following corollary.

\begin{cor}\label{C:homogalt}
If $E$ is homogeneous of degree $\vec{\jmath}$, then for all $\vec{k}<\vec{j}$ and for all $\vec{U}$, $T_{\vec{k}}E(\vec{U})$ is contractible.
\end{cor}

\begin{proof}
From \refT{multihomogclass} and \refL{tjrlemma} it follows that that if $E$ is homogeneous of degree $\vec{\jmath}$, then for all $\vec{k}<\vec{j}$ and for all $\vec{U}$, $T_{\vec{k}}E(\vec{U})$ is contractible. To establish this, it is sufficient to show that for all $\vec{k}<\vec{\jmath}$, $T_{\vec{k}}\holim_{R\neq\emptyset}T_{\vec{\jmath}_R}\Gamma$ is equivalent to $T_{\vec{k}}\Gamma$. Equivalently, it is enough to show that the $m$-cube $R\mapsto T_{\vec{k}}T_{\vec{\jmath}_R}\Gamma$ is homotopy cartesian.

For $\vec{k}<\vec{\jmath}_{\underline{m}}$, $T_{\vec{k}}T_{\vec{\jmath}_R}E\simeq T_{\vec{k}}E$ for all $R\neq\emptyset$ by \refP{iteratedvecT}, and it is clear that the $m$-cube in question is homotopy cartesian. For $\vec{k}\geq\vec{\jmath}_{\underline{m}}$, we must have $\vec{k}=\vec{\jmath}_S$ for some $S\subset\underline{m}$. By inspection, it is clear that $T_{\vec{\jmath}_S}T_{\vec{\jmath}_R}\Gamma\simeq T_{\vec{\jmath}_{R\cup S}}$. Then the $m$-cube $R\mapsto T_{\vec{\jmath}_{R\cup S}}\Gamma$ is homotopy cartesian as follows. For any $s\in S$, let $U$ range through subsets of $\underline{m}-\{s\}$. Then then $m$-cube $R\mapsto T_{\vec{\jmath}_{R\cup S}}\Gamma$ can be expressed as a map of $(m-1)$-cubes 
$$
(U\mapsto T_{\vec{\jmath}_{U\cup S}}\Gamma)\longrightarrow (U\mapsto T_{\vec{\jmath}_{U\cup\{s\}\cup S}}\Gamma)
$$
which are clearly identical.
\end{proof}


Before we embark on the proof of \refT{multihomogclass}, we need to discuss a few technical constructions.

\begin{defin}
For $\vec{\jmath}=(j_1,\ldots, j_m)$, and $\vec{U}=(U_1,\ldots, U_m)\in\mathcal{O}(\vec{P})$, let $\mathcal{I}^{(\vec{\jmath})}(\vec{U})$ denote the category whose objects are open sets $\vec{B}=(B_1,\ldots, B_m)$ such that $B_i$ is diffeomorphic to exactly $j_i$ disjoint open balls in $U_i$. The morphisms are the inclusion maps which are isotopy equivalences.
\end{defin}

\begin{lemma}\label{L:catconfig}
There is an equivalence, natural in $\vec{U}$, 
$$\abs{\mathcal{I}^{(\vec{\jmath})}(\vec{U})}\simeq {\vec{U}\choose \vec{\jmath}}.
$$
\end{lemma}

\begin{proof}
Adapt the proof of Lemma 3.5 of \cite{W:EI1}.
\end{proof}

Recall, that $\mathcal{Z}^m_{\leq \vec{\jmath}}$ is the subposet of $\mathcal{Z}^m$ consisting of all $\vec{k}$ such that $\vec{k}\leq\vec{\jmath}$, and that $\mathcal{Z}^m_{< \vec{\jmath}}$ is the subposet where $\vec{k}<\vec{\jmath}$. Recall \refD{jsubr}.

\begin{prop}\label{P:mcubehomog}
There is an equivalence of functors

$$\underset{\vec{k}<\vec{\jmath}}{\holim}\, T_{\vec{k}}F\simeq\underset{R\neq\emptyset}{\holim}\, T_{\vec{\jmath}_R}F.$$
\end{prop}

\begin{proof}
This follows immediately from \refL{posetideals}, as 
$\mathcal{Z}^{m}_{\leq\vec{\jmath}_{\{i\}}}$ are all clearly ideals whose union covers 
$\mathcal{Z}^{m}_{<\vec{\jmath}}$.
\end{proof}

\begin{rem}  As promised in Remark \ref{R:MultiLayerCube}, this proposition says that the $\vec{\jmath}^{th}$ layer of the Taylor multi-tower from \refD{MultiLayer} is equivalent to the total homotopy fiber of a certain cubical diagram.  This is in complete analogy to the discussion at the beginning of Section \ref{S:multihomog} regarding homogeneous polynomials in ordinary multivariable calculus.
\end{rem}

\begin{prop}\label{P:homogfibers}

Let $p:Z\rightarrow {\vec{P}\choose\vec{\jmath}}$ be a fibration. Let 
$$
E(\vec{U})=\hofiber\left(\Gamma{\vec{U}\choose\vec{\jmath}}\longrightarrow\underset{R\neq\emptyset}{\holim}\;\Gamma\left(\del_R{\vec{U}\choose\vec{\jmath}}\right)\right).
$$ 
Suppose $\vec{V}$ is a tubular neighborhood of $\vec{S}\subset\vec{P}$, where $\vec{S}$ contains $j_i$ elements in the $i^{th}$ variable. Then the composed map

$$E(\vec{V})\longrightarrow\Gamma{\vec{V}\choose\vec{\jmath}}\stackrel{eval_{\vec{S}}}{\longrightarrow}p^{-1}(\vec{S})$$

is an equivalence.
\end{prop}

\begin{proof}
This is a straightforward adaptation of \cite[Proposition 8.4]{W:EI1}.
\end{proof}

This proposition says that we can describe the fibers of the fibration in terms of $E$ as the values $E(\vec{V})$ for a tubular neighborhood $\vec{V}$ of $\vec{S}$ in $\vec{P}$.

The following useful technical lemma will be used in the proof of the classification theorem (\refT{multihomogclass}) to produce a space of sections functor from a homogeneous one.

\begin{lemma}[\cite{Dwyer:BG}, Lemma 3.12]\label{L:quasifiberlemma}
Suppose $X:\mathcal{C}\rightarrow \Spaces$ is a functor which takes all morphisms to homotopy equivalences. Then $\hocolim_{\mathcal{C}}X$ quasifibers over $\abs{\mathcal{C}}$ and the associated space of sections is homotopy equivalent to $\holim_{\mathcal{C}}X$
\end{lemma}

We are finally ready to prove the classification theorem for homogeneous functors.

\begin{proof}[Proof of \refT{multihomogclass}]
Given $E$ homogeneous of degree $\leq\vec{\jmath}$, define

$$F'(\vec{U})=\underset{\mathcal{I}^{(\vec{\jmath})}(\vec{U})}{\holim}\, E.$$

By \refL{quasifiberlemma}, $E:\mathcal{I}^{\vec{\jmath}}(\vec{U})\rightarrow\Spaces$ takes all morphisms to weak equivalences and hence $\hocolim_{\mathcal{I}^{(\vec{\jmath})}(\vec{U})} E$ quasifibers over $\abs{\mathcal{I}^{(\vec{\jmath})}(\vec{U})}$ and the section space of the associated fibration may be identified with $F'$. Note that this is a fibration over $\vec{U}\choose\vec{\jmath}$ by \refL{catconfig}, and a fibration is natural in $\vec{U}$. That is, a morphism $\vec{V}\subset\vec{U}$ induces a commutative (pullback) diagram
$$
\xymatrix{
\underset{\mathcal{I}^{(\vec{\jmath})}(\vec{V})}{\hocolim}\, E\ar[r]\ar[d] & \underset{\mathcal{I}^{(\vec{\jmath})}(\vec{U})}{\hocolim}\, E\ar[d]\\
\abs{\mathcal{I}^{(\vec{\jmath})}(\vec{V})}\ar[r] & \abs{\mathcal{I}^{(\vec{\jmath})}(\vec{U})}\\
}
$$
Since the equivalence $\abs{\mathcal{I}^{(\vec{\jmath})}(\vec{U})}\simeq{\vec{U}\choose\vec{\jmath}}$ from \refL{catconfig} is natural in $\vec{U}$, it follows that $F'$ is equivalent to a functor
$$
F(\vec{U})=\Gamma\left({\vec{U}\choose\vec{\jmath}},Z;p\right).
$$
Here $Z$ is the total space of the associated fibration $p$, and is equivalent to $\hocolim_{\mathcal{I}^{(\vec{\jmath})}(\vec{U})} E$. Replacing $E$ by an equivalent functor if necessary allows us to assume that there is a natural transformation $E\rightarrow F$. Now suppose $\vec{S}\in{\vec{P}\choose\vec{\jmath}}$, and let $\vec{V}$ be a tubular neighborhood of $\vec{S}$. Then the composition
$$
E(\vec{V})\rightarrow F(\vec{V})=\Gamma({\vec{V}\choose\vec{\jmath}},Z;p)\rightarrow p^{-1}(S)
$$
is an equivalence by \refP{homogfibers}. Now consider the commutative square
$$
\xymatrix{
E\ar[r]\ar[d] & F\ar[d]\\
\underset{\vec{k}<\vec{\jmath}}{\holim}\, T_{\vec{k}}E\ar[r] & \underset{\vec{k}<\vec{\jmath}}{\holim}\, T_{\vec{k}}F\\
}
$$
Recall that $\holim_{\vec{k}<\vec{\jmath}}T_{\vec{k}}E(\vec{U})$ is contractible for all $\vec{U}$ since $E$ is homogeneous. The left vertical fiber is then of course $E(\vec{V})$, and by \refP{mcubehomog} we may identify $\holim_{\vec{k}<\vec{\jmath}}T_{\vec{k}}F$ with $\holim_{R\neq\emptyset}T_{\vec{\jmath}_R}F$. Of course, since $F(\vec{V})=\Gamma({\vec{V}\choose\vec{\jmath}})$, \refL{tjrlemma} and \refP{sectionshomog} together tell us that the right vertical fiber is $\Gamma^c({\vec{V}\choose\vec{\jmath}})$. Thus to complete the proof it suffices to show that the map of vertical fibers is an equivalence. 

Since everything in sight is polynomial of degree $\leq\vec{\jmath}$, it further suffices to prove that the map of vertical fibers is an equivalence for all $\vec{V}\in\mathcal{O}_{\vec{\jmath}}$. If $\vec{V}\in\mathcal{O}_{\vec{k}}$ for $\vec{k}<\vec{\jmath}$, then $E(\vec{V})\simeq\ast$ and $\Gamma^c({\vec{V}\choose\vec{\jmath}})\simeq\ast$ since both are homogeneous. If $V$ has exactly $j_i$ connected components in the $i^{th}$ variable, then it is a tubular neighborhood of some $\vec{S}\subset\vec{P}$, and the map of fibers is an equivalence by \refL{tjrlemma} and \refP{homogfibers}.
\end{proof}


\subsection{The fibers of the classifying fibration}\label{S:fibersclasslink}


In this section we work out the fibers of the classifying fibration from \refT{multihomogclass} for a good functor $F$. These are important objects and should be thought of as the derivatives of $F$. That is, we have constructed a Taylor tower for $F$ consisting of polynomial approximations $T_{\vec{\jmath}}F$ just as one has a Taylor series (centered at $\vec{0}$) comprised of Taylor polynomials for an ordinary function $f:\R^m\rightarrow \R$. \refT{multihomogclass} says that there is an equivalence of functors 
$$
L_{\vec{\jmath}}F(\vec{U})\simeq\Gamma^c\left({{\vec{U}\choose \vec{\jmath}},Z;p}\right)
$$
for some fibration $p:Z\rightarrow{\vec{U}\choose\vec{\jmath}}$. The space ${\vec{U}\choose\vec{\jmath}}={U_1\choose j_1}\times\cdots\times{U_m\choose j_m}$ plays the role of $x_1^{j_1}\cdots x_m^{j_m}/j_1!\cdots j_m!$  from the Taylor expansion of an ordinary function $f$. The coefficient of this monomial is of course the partial derivative $\frac{\del^{|\vec{\jmath}|}}{\del^{\vec{\jmath}}\vec{x}}f(\vec{0})$, and its analog is the fiber of $p$ over $\vec{S}\in{\vec{U}\choose\vec{\jmath}}$.  A good symbol for this fiber thus might be $$\frac{\del^{|\vec{\jmath}|}}{\del^{\vec{\jmath}}\vec{U}}F(\vec{\emptyset}).
$$
In \refT{fibersclassfib} below, we give an explicit description of these spaces for every $\vec{\jmath}$. This result follows easily from \cite[Proposition 9.1]{W:EI1}. We will follow the proof of \refT{fibersclassfib} with a remark about what it tells us about $
\frac{\del^{|\vec{\jmath}|}}{\del^{\vec{\jmath}}\vec{U}}\Link(\vec{S};N)
$. We will also compute by hand the fibers of the classifying fibration for $L_{(2,1)}\Link$ to give the reader an idea of what kind of arguments go into this without appealing to results of Weiss.

\begin{thm}\label{T:fibersclassfib}
Let $\vec{U}=(U_1,\ldots, U_m)$, $\vec{\jmath}=(j_1,\ldots, j_m)$, and $\vec{S}=(S_1,\ldots, S_m)\in{\vec{P}\choose\vec{\jmath}}$. Then the fiber over $\vec{S}\in{\vec{P}\choose\vec{\jmath}}$ of the classifying fibration for $L_{\vec{\jmath}}F$ is the total homotopy fiber of the $|\vec{\jmath}|$-cube

$$
T=T_1\coprod \cdots\coprod T_m\longmapsto F(V_{T_1},\ldots, V_{T_m})
$$
where the $T_i$ range through subsets of $S_i$ for each $i$, and $V_{T_i}$ is a tubular neighborhood of $T_i$ in $P_i$ obtained from a tubular neighborhood $V_{S_i}$ of $S_i$ by including only those components containing elements of $T_i$.
\end{thm}

\begin{proof}
Let $k=|\vec{\jmath}|$. By \refP{homogdisjoint}, 
$$
L_kF\simeq\Gamma^c\left({U\choose k},Z;p\right)=\prod_{|\vec{\jmath}|=k}\Gamma^c\left({\vec{U}\choose \vec{\jmath}},Z;p\right),
$$
where $p:Z\rightarrow{U\choose k}$ is the classifying fibration (here we regard $F$ as a functor of the single variable $U=U_1\coprod \cdots\coprod U_m$). It follows from the proof of \refT{twotowers} that the right-hand side is equivalent to the product over $|\vec{\jmath}|=k$ of $L_{\vec{\jmath}}F$.

By \cite[Proposition 9.1]{W:EI1}, the fiber over $S'\in{U_1\coprod\cdots\coprod U_m\choose k}$ is
$$
\tfiber(T\mapsto F(V_T)),
$$
where $T$ ranges through subsets of $S$, and $V_T$ is a tubular neighborhood of $T$ obtained from a tubular neighborhood of $V_S$ by only including those components of $V_S$ containing elements of $T$. Of course, as we noted in the proof of \refP{homogdisjoint},
$$
{U_1\coprod\cdots\coprod U_m\choose k}=\coprod_{|\vec{\imath}|=k}{\vec{U}\choose\vec{\imath}},
$$
and if $S\in{\vec{U}\choose\vec{\jmath}}$, we may write $S=S_1\coprod \cdots\coprod S_m$, $\vec{S}=(S_1,\ldots, S_m)$ for $S_i\subset U_i$, and the tubular neighborhood $V_S=V_{S_1}\coprod\cdots\coprod V_{S_m}$, and put $\vec{V}_{\vec{S}}=(V_{S_1},\ldots, V_{S_m})$. Thus the fiber over $\vec{S}\in{\vec{U}\choose\vec{\jmath}}$ is 
$$
\tfiber(\vec{T}\mapsto F(\vec{V}_{\vec{T}})),
$$
where $\vec{T}=(T_1,\ldots, T_m)$ ranges through subsets of $\vec{S}$. Rewriting this using the above shows that it is the total homotopy fiber of the cube
$$
T=T_1\coprod \cdots\coprod T_m\longmapsto F(V_{T_1},\ldots, V_{T_m}).
$$
\end{proof}

\begin{rem}
In case $F=\Link$, the tubular neighborhood  can be disregarded altogether, since $\Link(\vec{V}_{\vec{T}})\simeq\Link(\vec{T})$. Thus, the fibers of the classifying fibration for $L_{\vec{\jmath}}\Link$ are the total homotopy fibers of the $|\vec{S}|$-cube
$$
\vec{T}=(T_1,\ldots, T_m)\longmapsto\Link(T_1,\ldots, T_m;N).
$$
The spaces $\Link(T_1,\ldots, T_m;N)$ are ``partial configuration spaces'' in the sense that we may regard $\Link(T_1,\ldots, T_m;N)$ as a subset of $N^{|\vec{S}|}-D$, where $D$ is a union of some of the diagonals of $N^{|\vec{S}|}$. It is not hard to work out which diagonals in general; we will give an example in what follows.
\end{rem}

\begin{example}
Here we compute ``by hand'' the fibers of the classifying fibration for 
$$
(U_1,U_2)\longmapsto L_{(2,1)}\Link(U_1,U_2;N),
$$ 
where $(U_1,U_2)\in\mathcal{O}(P_1)\times\mathcal{O}(P_2)$. As in the proof of \refT{multihomogclass}, $L_{(2,1)}\Link(P_1,P_2;N)$ is equivalent to the space of sections of a fibration $p:Z\rightarrow{P_1\choose 2}\times P_2$ whose fiber over $(S_1,S_2)\in{P_1\choose 2}\times P_2$ is $L_{(2,1)}\Link(V_1,V_2;N)$.  Here $(V_1,V_2)$ is a tubular neighborhood of $(S_1,S_2)$. We will write $S_1=\{a_1,a_2\}$,  $S_2=\{b\}$, and $V_1=V_{a_1}\coprod V_{a_2}$.

Also writing $\Link(\vec{V})$ in place of $\Link(V_1,V_2;N)$, we have by \refP{mcubehomog} an equivalence

\begin{equation}\label{L21exeqn}
L_{(2,1)}\Link(\vec{V})\simeq\tfiber\left(
\raisebox{1cm}{
\xymatrix{
  T_{(2,1)}\Link(\vec{V}) \ar[r]\ar[d]  & T_{(1,1)}\Link(\vec{V})  \ar[d]\\
  T_{(2,0)}\Link(\vec{V}) \ar[r]  & T_{(1,0)}\Link(\vec{V})
}}
\right)
\end{equation}

We also have the following equivalences which follow from \refD{jMultiStage} and the remark immediately following. 
\begin{eqnarray*}
T_{(2,1)}\Link(V_1,V_2)&\simeq&\Link(V_1,V_2),\\
T_{(2,0)}\Link(V_1,V_2)&\simeq&\Link(V_1,\emptyset),\\
T_{(1,1)}\Link(V_1,V_2)&\simeq&\holim(\Link(V_{a_1},V_2)\rightarrow\Link(\emptyset,V_2)\leftarrow\Link(V_{a_2},V_2)),\mbox{ and}\\
T_{(1,0)}\Link(V_1,V_2)&\simeq&\holim(\Link(V_{a_1},\emptyset)\rightarrow\Link(\emptyset,\emptyset)\leftarrow\Link(V_{a_2},\emptyset)).
\end{eqnarray*}
 Observe also that $\Link(V_1,V_2)\simeq\Link(a_1\cup a_2,b)$, and likewise for the other spaces. This fact is special to this functor. It is not true, for example, for spaces of embeddings; in that case we cannot ignore tangential information. In general, then, we cannot simplify expressions like $F(V_1,V_2)$ as we have done with $\Link(V_1,V_2)$ above. Other than this, everything else will carry through without further ado.
 
Now consider the following diagram

$$
\xymatrix{
&&\Link(a_1\cup a_2,\emptyset)\ar[d]\ar[dl]\ar[dr]&&\\
&\Link(a_1,\emptyset)\ar[r]&\Link(\emptyset,\emptyset)&\Link(a_2,\emptyset)\ar[l]&\\
\Link(a_1,b)\ar[rr]\ar[ru]&&\Link(\emptyset,b)\ar[u]&&\Link(a_2,b)\ar[ll]\ar[lu]
}
$$

From the observations above, taking homotopy limits along rows yields

$$
\xymatrix{
T_{(2,0)}\Link(a_1\cup a_2,b)\ar[d]\\
T_{(1,0)}\Link(a_1\cup a_2,b)\\
T_{(1,1)}\Link(a_1\cup a_2,b)\ar[u]
}
$$

The homotopy limit of this diagram is the target of the map from equation \eqref{L21exeqn}. Hence by inspection, $L_{(2,1)}(V_1,V_2)$ is equivalent to the total homotopy fiber of the cube
$$
\xymatrix@=20pt{
   \Link(a_1\cup a_2,b)\ar[rr]\ar[dd]\ar[dr]      &           &   \Link(a_1\cup a_2,\emptyset) \ar'[d][dd]           \ar[dr]  &                  \\
        &  \Link(a_2,b) \ar[rr] \ar[dd]  &             & \Link(a_2,\emptyset)
         \ar[dd] \\
\Link(a_1,b) \ar'[r][rr] \ar[dr] &        &   \Link(a_1,\emptyset)
\ar[dr] &                   \\
        &   \Link(\emptyset,b) \ar[rr]      &                    &
\Link(\emptyset,\emptyset)
}
$$

Note, for example, that there is an equivalence given by the evaluation map 
$$
\Link(a_1\cup a_2,b;N)\longrightarrow N^3-(\Delta_{13}\cup\Delta_{23})
$$
which sends $f$ to $(f(a_1),f(a_2),f(b))$. Here $\Delta_{ij}$ is the subset of $(x_1,x_2,x_3)\in N^3$ for which $x_i=x_j$. This is what we referred to above as a partial configuration space, for obvious reasons. The cube above happens to be homotopy cartesian. It is not clear how cartesian more general cubes of link maps are. In \cite{MV:Links} we will explore the cohomology of the total fibers of diagrams such as those above and relate this to classical link invariants.
\end{example}

As mentioned at the beginning of this section, the fibers of the classifying fibration should be thought of as the derivatives of the functor in question, and so, if only by analogy, it should be useful to understand them better. For link maps, a basic question to which we do not know the answer is how highly connected these spaces are. It is not difficult to work out some special cases, but in general this appears to be a difficult problem.

\section{Relationship between single variable and multivariable polynomial functors and Taylor towers}\label{S:TowersRelation}

We are now ready to exhibit a relationship between stages of the single variable and multivariable Taylor towers. Recall the setup: $F$ is a good functor from $\mathcal{O}(\coprod_iP_i)$ to topological spaces, which we may also regard as a functor from $\mathcal{O}(\vec{P})$ to topological spaces, where $\vec{P}=(P_1,\ldots, P_m)$. Our proof requires the classification  theorem for homogeneous functors but it would be interesting to know if there is one that does not. In particular, this proof requires a choice of a basepoint in $F(\vec{P})$, which may or may not be available depending on $F$ and $\vec{P}$.


\begin{thm}\label{T:twotowers}
There is an equivalence
$$
T_kF\stackrel{\simeq}{\longrightarrow}\underset{|\vec{\jmath}|\leq k}{\holim}\, T_{\vec{\jmath}}F.
$$
\end{thm}

\begin{proof}
We induct on $k$. The cases $k=0$ and $k=1$ are easy to see on the categorical level. For $k=0$, we have $\mathcal{O}_0(U_1\coprod \cdots\coprod U_m)=\prod_i\mathcal{O}_0(U_i)$, which consist of $\emptyset$ and $\vec{\emptyset}$ respectively. For $k=1$, we have an equivalence $\mathcal{O}_1(U_1\coprod \cdots\coprod U_m)=\coprod_i\mathcal{O}_1(U_i)$, and since homotopy limits turn coproducts into products, we have the desired result by inspection.

Now assume the result is true for all $l<k$. Consider the following diagram whose columns are homotopy fiber sequences.
$$
\xymatrix{
L_kF\ar[r]\ar[d] & L\ar[d]\\
T_kF\ar[r]\ar[d] & \underset{|\vec{\jmath}|\leq k}{\holim}\, T_{\vec{\jmath}}F\ar[d]\\
T_{k-1}F\ar[r] & \underset{|\vec{\jmath}|\leq k-1}{\holim}\, T_{\vec{\jmath}}F\\
}
$$
Here $L=\hofiber(\holim_{|\vec{\jmath}|\leq k}T_{\vec{\jmath}}F\rightarrow\holim_{|\vec{\jmath} |\leq k-1}T_{\vec{\jmath}}F)$. Recall from \refP{homogdisjoint} that $L_kF$ is the product, over $|\vec{\jmath}|=k$, of $L_{\vec{\jmath}}$. It is enough to show that $L_kF\simeq L$.

Recall that $\mathcal{Z}^m_{\leq k}$ is the poset of tuples $\vec{\jmath}=(j_1,\ldots, j_m)$ such that $|\vec{\jmath}|=j_1+\cdots + j_m\leq k$. We have a covering $\mathcal{Z}^m_{\leq k}=\cup_{|\vec{\jmath}|= k}\mathcal{Z}^m_{\leq\vec{\jmath}}$, and the $\mathcal{Z}^m_{\leq\vec{\jmath}}$ are ideals. For a subset $S\subset \{\vec{\jmath}\ |\ |\vec{\jmath}|=k\}$, let $\mathcal{Z}^m_{\leq\vec{\jmath}_S}=\cap_{\vec{\jmath}\in S}\mathcal{Z}^m_{\vec{\jmath}}$. Here $\vec{\jmath}_S=(a_1,\ldots, a_m)$, where $a_i=\min\{j_i\  | \ \vec{\jmath}=(j_1,\ldots, j_m)\in S\}$. By \refL{posetideals} we have an equivalence

$$\underset{|\vec{\jmath}|\leq k}{\holim}\, T_{\vec{\jmath}}F\simeq\underset{S\neq\emptyset}{\holim}\,\underset{\vec{\jmath}\leq\vec{\jmath}_S}{\holim}\, T_{\vec{\jmath}}F\simeq\underset{S\neq\emptyset}{\holim}\, T_{\vec{\jmath}_S}F.$$

The second is an equivalence follows because $\vec{\jmath}_S$ is a final object in $\mathcal{Z}^m_{\leq\vec{\jmath}_S}$. 

We also have a similar but different decomposition  $\mathcal{Z}^m_{\leq k-1}=\cup_{|\vec{\jmath}|=k}\mathcal{Z}^m_{<\vec{\jmath}}$. This is once again a covering by ideals, but note that if $S\subset \{\vec{\jmath}\  |\  |\vec{\jmath}|=k\}$, we have

\begin{equation*}
\cap_{\vec{\jmath}\in S}\mathcal{Z}^m_{<\vec{\jmath}}=
\begin{cases}
\mathcal{Z}^m_{<\vec{\jmath}}, &\text{if $S=\{\vec{\jmath}\}$};\\
\mathcal{Z}^m_{\leq\vec{\jmath}_S}, &\text{if $|S|>1$.}
\end{cases}
\end{equation*}

The categories $\mathcal{Z}^m_{\leq\vec{\jmath}_S}$ each have a final object, namely $\vec{\jmath}_S$, and so $\holim_{\vec{k}\in\mathcal{Z}^m_{\leq\vec{\jmath}_S}}T_{\vec{k}}F\simeq T_{\vec{\jmath}_S}$. By \refL{posetideals} we have an equivalence

$$\underset{|\vec{\jmath}|\leq k-1}{\holim}\, T_{\vec{\jmath}}F\simeq\underset{S\neq\emptyset}{\holim}\,\underset{\substack{ \vec{\jmath}<\vec{\imath} \mbox{ if }S=\{\vec{\imath}\} \\ \vec{\jmath}\leq\vec{\jmath}_S\mbox{ if }|S|>1 }}{\holim}\, T_{\vec{\jmath}}F.$$

By inspection, it follows that

$$L=\hofiber\left(\underset{|\vec{\jmath}|\leq k}{\holim}\, T_{\vec{\jmath}}F\rightarrow\underset{|\vec{\jmath}|\leq k-1}{\holim}\, T_{\vec{\jmath}}F\right)=\underset{|\vec{\jmath}|=k}{\prod}\,L_{\vec{\jmath}}F,$$

and, in light of \refP{homogdisjoint}, it is clear that $L_kF\rightarrow L$ is an equivalence.

\end{proof}

\section{Non-functorial finite models for stages of the Taylor tower}\label{S:FiniteModels}


The stages $T_kF(U)$ and $T_{\vec{\jmath}}F(\vec{U})$ of single variable and multivariable Taylor towers given in Sections \ref{S:Polynomials} and \ref{S:MultiPolynomials} have the advantage of being functors of $U$ and $\vec{U}$ respectively, but the homotopy limits involved can be unwieldy because the categories $\mathcal{O}_k(U)$ and $\mathcal{O}_{\vec{\jmath}}(\vec{U})$ are rather large. At the expense of losing functoriality, we can sometimes describe the stages in terms of a homotopy limit of values of $F$ itself over a finite category. This is the case for the examples we will be most concerned with in \cite{MV:Links}, namely links and homotopy links.  We will state the definitions in terms of $T_{\vec{\jmath}}F(\vec{U})$, but everything specializes in a straightforward way to $T_kF(U)$.

Suppose that $P=\coprod_{i=1}^m I$ is a disjoint union of neat smooth submanifolds of a manifold $N$ with boundary $\del N$, each diffeomorphic with an interval. Let $\mathcal{O}(P_i)$  denote the poset of open subsets of $P_i$ which contain $\partial P_i$.


\begin{prop}\label{P:FiniteTower}
Let $F=\Emb_\partial(-,N)$ or $\Link_\partial(-,N)$, and suppose $P=\coprod P_i$ is as above; that is, each $P_i=I$. Let $\vec{\jmath}=(j_1,\ldots, j_m)\geq\vec{-1}$ be a tuple of nonnegative integers, and $A^i_0,\ldots A_{j_i}^i$ be pairwise disjoint nonempty closed connected subintervals of the interior of $P_i$ for every $i$. Let $P_{i,S_i}=P_i-\cup_{k\in S_i} A^i_{k}$, where $S_i$ ranges through subsets of $\{0,1,\ldots, j_i\}$, and put $\vec{P}_{\vec{S}}=(P_{S_1},\ldots, P_{S_m})$. Then
$$
T_{\vec{\jmath}}F(\vec{P})\simeq \underset{\vec{S}\in\mathcal{P}_0([\vec{\jmath}])}{\holim}\, \, F(\vec{P}_{\vec{S}}).
$$
\end{prop}

\begin{proof}
Since $T_{\vec{\jmath}}F$ is a polynomial of degree $\leq\vec{\jmath}$, it follows that 
$$
T_{\vec{\jmath}}F(\vec{P})\simeq \underset{\vec{S}\in\mathcal{P}_0([\vec{\jmath}])}{\holim}\, \, T_{\vec{\jmath}}F(\vec{P}_{\vec{S}}).
$$
Since $P_{i,S_i}$ is diffeomorphic to at most $j_i$ open connected intervals (ignoring the boundary, where everything is already fixed) when $\vec{S}\neq\vec{\emptyset}$, we have an equivalence 
$$
F(\vec{P}_{\vec{S}})\simeq T_{\vec{\jmath}}F(\vec{P}_{\vec{S}})
$$ since $\vec{P}_{\vec{S}}$ is a final object in $\mathcal{O}_{\vec{\jmath}}(\vec{P}_{\vec{S}})$. This completes the proof.
\end{proof}

\begin{rem}
The case when $F=\Emb$ and  $m=1$ has been extensively studied.  More details can be found in \cite{S:TSK}.  In that paper, Sinha uses pairwise disjoint nonempty \emph{open} connected subintervals of $I$ for technical reasons (he requires compactness of the complement). The trouble with using open sets to delete is that their complements are closed, and these would obviously not be elements in our category of open sets. However, this is not a problem in this case for two reasons:
\begin{enumerate}
\item the model for $T_{\vec{\jmath}}F$ given in Proposition \ref{P:FiniteTower} is non-functorial, and 
\item the homotopy type of $F(\vec{P}_{\vec{S}})$ is independent of whether the intervals are open or closed. 
\end{enumerate}
We will need to use open sets for the embedding functor to generalize various results of Sinha's  to $F=\Emb, \Link$ and $m>1$ in \cite{MV:Links}.\end{rem}



%




%

\bibliographystyle{amsplain}

\bibliography{/Users/ivolic/Desktop/Papers/Bibliography}

\def\cprime{$'$} \def\cprime{$'$}
\providecommand{\bysame}{\leavevmode\hbox to3em{\hrulefill}\thinspace}
\providecommand{\MR}{\relax\ifhmode\unskip\space\fi MR }
\providecommand{\MRhref}[2]{%
  \href{http://www.ams.org/mathscinet-getitem?mr=#1}{#2}
}
\providecommand{\href}[2]{#2}
\begin{thebibliography}{10}

\bibitem{ALV}
Gregory Arone, Pascal Lambrechts, and Ismar Voli{\'c}, \emph{Calculus of
  functors, operad formality, and rational homology of embedding spaces}, Acta
  Math. \textbf{199} (2007), no.~2, 153--198.

\bibitem{BK}
A.~K. Bousfield and D.~M. Kan, \emph{Homotopy limits, completions and
  localizations}, Springer-Verlag, Berlin, 1972, Lecture Notes in Mathematics,
  Vol. 304.

\bibitem{CR:LinkingInv}
Vladimir~V. Chernov and Yuli~B. Rudyak, \emph{Toward a general theory of
  linking invariants}, Geom. Topol. \textbf{9} (2005), 1881--1913 (electronic).

\bibitem{Dwyer:BG}
W.~G. Dwyer, \emph{The centralizer decomposition of {$BG$}}, Algebraic
  topology: new trends in localization and periodicity ({S}ant {F}eliu de
  {G}u\'\i xols, 1994), Progr. Math., vol. 136, Birkh\"auser, Basel, 1996,
  pp.~167--184.

\bibitem{CalcII}
Thomas~G. Goodwillie, \emph{Calculus {II}: {A}nalytic functors}, $K$-Theory
  \textbf{5} (1991/92), no.~4, 295--332.

\bibitem{GK}
Thomas~G. Goodwillie and John~R. Klein, \emph{Multiple disjunction for spaces
  of smooth embeddings}, in preparation.

\bibitem{GKW}
Thomas~G. Goodwillie, John~R. Klein, and Michael~S. Weiss, \emph{A {H}aefliger
  style description of the embedding calculus tower}, Topology \textbf{42}
  (2003), no.~3, 509--524.

\bibitem{GM:LinksEstimates}
Thomas~G. Goodwillie and Brian~A. Munson, \emph{A stable range description of
  the space of link maps}, Algebr. Geom. Topol. \textbf{10} (2010), 1305--1315.

\bibitem{GW:EI2}
Thomas~G. Goodwillie and Michael Weiss, \emph{Embeddings from the point of view
  of immersion theory {II}}, Geom. Topol. \textbf{3} (1999), 103--118
  (electronic).

\bibitem{HabLin-Classif}
Nathan Habegger and Xiao-Song Lin, \emph{The classification of links up to
  link-homotopy}, J. Amer. Math. Soc. \textbf{3} (1990), no.~2, 389--419.

\bibitem{HQ:Bordism}
Allen Hatcher and Frank Quinn, \emph{Bordism invariants of intersections of
  submanifolds}, Trans. Amer. Math. Soc. \textbf{200} (1974), 327--344.

\bibitem{Kosch-Milnor}
Ulrich Koschorke, \emph{A generalization of {M}ilnor's {$\mu$}-invariants to
  higher-dimensional link maps}, Topology \textbf{36} (1997), no.~2, 301--324.

\bibitem{Kos:DiffMflds}
Antoni~A. Kosinski, \emph{Differential manifolds}, Pure and Applied
  Mathematics, vol. 138, Academic Press Inc., Boston, MA, 1993.

\bibitem{LTV:Vass}
Pascal Lambrechts, Victor Turchin, and Ismar Voli\'{c}, \emph{The rational
  homology of spaces of long knots in codimension $>2$}, to appear in Geometry
  \& Topology, arXiv:math.AT/0703649.

\bibitem{Milnor-Mu}
John Milnor, \emph{Link groups}, Ann. of Math. (2) \textbf{59} (1954),
  177--195.

\bibitem{M:Emb}
Brian~A. Munson, \emph{Embeddings in the {$3/4$} range}, Topology \textbf{44}
  (2005), no.~6, 1133--1157.

\bibitem{M:LinkNumber}
\bysame, \emph{A manifold calculus approach to link maps and the linking
  number}, Algebr. Geom. Topol. \textbf{8} (2008), no.~4, 2323--2353.

\bibitem{MV:Links}
Brian~A. Munson and Ismar Voli\'c, \emph{Cosimplicial models for spaces of
  links}, submitted.

\bibitem{ST:HOIN}
Rob Schneiderman and Peter Teichner, \emph{Higher order intersection numbers of
  2-spheres in 4-manifolds}, Algebr. Geom. Topol. \textbf{1} (2001), 1--29
  (electronic).

\bibitem{Scott:HoLinks}
G.~P. Scott, \emph{Homotopy links}, Abh. Math. Sem. Univ. Hamburg \textbf{32}
  (1968), 186--190.

\bibitem{S:OKS}
Dev~P. Sinha, \emph{Operads and knot spaces}, J. Amer. Math. Soc. \textbf{19}
  (2006), no.~2, 461--486 (electronic).

\bibitem{S:TSK}
\bysame, \emph{The topology of spaces of knots: cosimplicial models}, Amer. J.
  Math. \textbf{131} (2009), no.~4, 945--980.

\bibitem{Skop:Massey-Rolfsen}
A.~Skopenkov, \emph{On the generalized {M}assey-{R}olfsen invariant for link
  maps}, Fund. Math. \textbf{165} (2000), no.~1, 1--15.

\bibitem{V:FTK}
Ismar Voli{\'c}, \emph{Finite type knot invariants and the calculus of
  functors}, Compos. Math. \textbf{142} (2006), no.~1, 222--250.

\bibitem{W:EmbCalc}
Michael Weiss, \emph{Calculus of embeddings}, Bull. Amer. Math. Soc. (N.S.)
  \textbf{33} (1996), no.~2, 177--187.

\bibitem{W:EI1}
\bysame, \emph{Embeddings from the point of view of immersion theory {I}},
  Geom. Topol. \textbf{3} (1999), 67--101 (electronic).

\end{thebibliography}

\end{document}